\documentclass[a4paper,10pt,leqno]{amsart}
\title[Segal's spectral sequence in twisted equivariant K-theory]
{Segal's spectral sequence in twisted equivariant K-theory for proper and discrete actions}

\author{No\'{e} B\'{a}rcenas }
        \address{Centro  de  Ciencias Matem\'aticas. \\ Universidad  Nacional Aut\'onoma  de  M\'exico \\ Apartado Postal  61-3 Xangari.  Morelia, Michoac\'an.  M\'exico 58089}
         \email{barcenas@matmor.unam.mx}
         \urladdr{http://matmor.unam.mx/~barcenas/}

\author{Jes\'us  Espinoza}
        \address{Centro  de  Ciencias Matem\'aticas. \\ Universidad  Nacional Aut\'onoma  de  M\'exico \\ Apartado Postal  61-3 Xangari.  Morelia, Michoac\'an.  M\'exico 58089}
        \email{jespinoza@matmor.unam.mx}    

 \author{Bernardo Uribe}
\address{Departamento de Matem\'aticas y Estad\'istica\\ Universidad del Norte\\ Km 5 V\'ia Puerto Colombia\\ Barranquilla\\ Colombia}
\email{bjongbloed@uninorte.edu.co}

\author{Mario Vel\'asquez}
\address{Departamento de Matem\'aticas. \\ Pontificia Universidad Javeriana\\ Cra. 7 No. 43-82 - Edificio Carlos Ort\'iz, 5to piso\\ Bogot\'a D.C., Colombia}
\email{mario.velasquez@javeriana.edu.co}
\urladdr{https://sites.google.com/site/mavelasquezm/}

\subjclass[2010]{(primary) 19L47, 19L50, 55N91, (secondary) 18G40}
\keywords{Twisted equivariant K-theory, Bredon cohomology, proper actions, twisted representations.}

\usepackage{pdfsync}
\usepackage{hyperref}
\usepackage{calc}
\usepackage{enumerate,amssymb}
\usepackage{latexsym, amscd}
\usepackage{color}
\usepackage[arrow,curve,matrix,tips,2cell]{xy}
  \SelectTips{eu}{10} \UseTips
  \UseAllTwocells
\usepackage{tikz}
\usepackage{pdfcolmk}

\DeclareMathAlphabet\EuR{U}{eur}{m}{n}
\SetMathAlphabet\EuR{bold}{U}{eur}{b}{n}

\makeindex             

\theoremstyle{plain}
\newtheorem{theorem}{Theorem}[section]
\newtheorem{lemma}[theorem]{Lemma}
\newtheorem{proposition}[theorem]{Proposition}

\theoremstyle{definition}
\newtheorem{definition}[theorem]{Definition}

\newtheorem{remark}[theorem]{Remark}

{\catcode`@=11\global\let\c@equation=\c@theorem}

\newcommand{\comsquare}[8]                   
{\begin{CD}
#1 @>#2>> #3\\
@V{#4}VV @V{#5}VV\\
#6 @>#7>> #8
\end{CD}
}

\newcommand{\xycomsquare}[8]                   
{\xymatrix
{#1 \ar[r]^{#2} \ar[d]^{#4} &
#3 \ar[d]^{#5}  \\
#6\ar[r]^{#7} &
#8
}
}

\newcommand{\calc}{\mathcal{C}}
\newcommand{\cald}{\mathcal{D}}

\newcommand{\calfin}{\mathcal{FIN}}

\newcommand{\calh}{\mathcal{H}}

\newcommand{\calu}{{\mathcal U}}

\newcommand{\calall}{\mathcal{ALL}}

\newcommand{\IH}{{\mathbb H}}

\newcommand{\IZ}{{\mathbb Z}}

\newcommand{\curs}{\EuR}

\newcommand{\SETS}{\curs{SETS}}

\newcommand{\coker}{\operatorname{coker}}
\newcommand{\colim}{\operatorname{colim}}

\newcommand{\hocolim}{\operatorname{hocolim}}

\newcommand{\id}{\operatorname{id}}

\newcommand{\cover}{\mathcal{N}_{G} \mathcal{U}}

\newcommand{\ktheory}[3]{K_{#1}^{#2}(#3, P)}
\newcommand{\Fred}{\ensuremath{{\mathrm{Fred}}}}
\newcommand{\Funct}{\mathrm{Funct}}

\newcommand{\OO}{\mathcal{O}}
\newcommand{\XX}{\mathcal{X}}
\newcommand{\YY}{\mathcal{Y}}

\newcommand{\HH}{\mathcal{H}}

\def\Hom{\mathrm{Hom}}

\newcommand{\eub}[1]{\underline{E}#1}              
\newcommand{\B}{{\rm{ B}}}
\newcommand{\higherlim}[3]{{\setbox1=\hbox{\rm lim}
        \setbox2=\hbox to \wd1{\leftarrowfill} \ht2=0pt \dp2=-1pt
        \mathop{\vtop{\baselineskip=5pt\box1\box2}}
        _{#1}}^{#2}#3}

\newcommand{\version}[1]                       
{\begin{center} last edited on #1\\
last compiled on \today\\
name of texfile: \jobname
\end{center}
}

\newcounter{commentcounter}

\begin{document}

\typeout{----------------------------  linluesau.tex  ----------------------------}


\typeout{------------------------------------ Abstract ----------------------------------------}

\begin{abstract}
We use a spectral sequence developed by  Graeme Segal in order to understand the twisted $G$-equivariant  K-theory for proper
and discrete actions. We show that the second page of this spectral sequence is isomorphic to a version of  Bredon cohomology with  local coefficients in twisted representations. We  furthermore explain  some  phenomena  concerning the  third  differential of the spectral sequence, and we recover known results when the twisting comes from finite  order  elements  in discrete torsion.
  \end{abstract}

\maketitle

\section*{Introduction}

One of the tools  for  calculating  generalized cohomology groups  is the  Atiyah-Hirzebruch  spectral  sequence  which
was originally developed in \cite{AtiyahHirzebruch} in order to study K-theory. Many generalizations of this spectral sequence have been developed
for studying cohomology theories in the equivariant context and we will pay specific attention to the spectral sequence developed by Segal in \cite{Segal}.

Twisted  equivariant  K-theory was  defined  by  Atiyah  and  Segal  in \cite{atiyahsegal} using  bundles of Fredholm operators and  was extended  to  the context  of  proper  actions by M. Joachim and the first three authors in \cite{barcenasespinozajoachimuribe}.  Due  to  the  relation   of     the  Verlinde  Algebra  to  the  twisted  equivariant  K-theory  of  a  compact Lie  group acting  on itself  by  conjugation  stablished  by Freed, Hopkins  and  Teleman in \cite{freedhopkinsteleman3},    computational methods  have  become  necessary in order to calculate the twisted equivariant K-theory groups.  Specialized  to  the  case of  the  conjugation action,   the  K\"unneth  Spectral  sequence  \cite{braun}, and  the  Rothenberg-Steenrod  spectral  sequence for  twisted  equivariant  K-homology \cite{douglas}  have  been successfully  used to  determine  the  twisted  equivariant  K-theory  groups. Theoretical tools  like  the  completion  theorem   of  Lahtinen \cite{lahtinen},  together  with  the previously  described   methods,   define the  group  of  ideas used  for  the  computation  in  \cite{kneezel} of  the  twisted  K-theory  of  the  loop  space  of  the  classifying  space of  a  simply  connected and simple compact  lie  group. 

Besides  from  notable  specific  examples  explained  in \cite{freedhopkinsteleman1},  and  the  case when the twisting comes from  discrete  torsion \cite{Dwyer}  (where a method  is  used  to  reduce  the  construction  of  the  spectral  sequence  to  an  untwisted  version  as done in \cite{lueckdavis}), a systematic study of a  spectral  sequence  for  computing twisted  equivariant K-theory  under  the  presence  of a generic twist has not been carried out.  This is the main objective of this work.

In \cite{barcenasespinozajoachimuribe} the  twisted  equivariant  K-theory for     a  proper  $G$-ANR  $X$  was  defined  given  
a Stable equivariant   Projective Unitary   bundle;  these bundles were shown  in   \cite{barcenasespinozajoachimuribe} to be  classified  by elements  of  the degree three  Borel  cohomology group $H^{3}(X {\times}_GEG;\mathbb{Z)}$.   In  this  note  we  use  the  explicit construction of the universal stable equivariant projective unitary bundle done in  \cite{barcenasespinozajoachimuribe}  in order to determine the first two pages of Segal's spectral sequence converging to the twisted equivariant K-theory groups. For this purpose we  develop  a twisted version of Bredon  cohomology,  cohomology which turns out to determine the $E_{2}$-page  of Segal's  spectral  sequence once it is applied to an equivariantly contractible cover.  

The  construction  of the  spectral  sequence extends and  generalizes previous work  of  C. Dwyer \cite{Dwyer}, who only treated the twistings which are classified by  cohomology classes of  finite  order which lie in the image of the canonical map $H^3(BG,\mathbb Z) \to H^3( X{\times}_G{ E }G;\mathbb{Z)}$; these twistings take the name of discrete torsion twistings. 

The main result of this note, which is Theorem \ref{theoremspectralsequence}, relies on the construction and the properties of the universal stable equivariant projective unitary bundle carried out in \cite{barcenasespinozajoachimuribe}. Since this work
can be seen as a continuation of what has been done in \cite{barcenasespinozajoachimuribe}, we will use the notation, the definitions and the results of that paper. We will not reproduce any proof that already appears in \cite{barcenasespinozajoachimuribe}, instead we will give appropriate references whenever a definition or a result of \cite{barcenasespinozajoachimuribe}  is used. 

We emphasize that  the topological issues that may appear
when working with the Projective Unitary Group have all been resolved in \cite[Section 15]{LueckUribe}
when it is endowed with the norm topology. We therefore assume in this work that  we are working with the norm
topology when discussing topological properties of operator spaces.

This  note  is  organized  as  follows. In  Section  \ref{sectionbredon} a  version  of  Bredon  cohomology  associated  to  an  equivariant  cover  of  a  space is  constructed. In Section \ref{sectionOGspaces}  the  basics of  Transformation  Groups  and  Parametrized  Homotopy Theory  needed for  the  construction are  quickly  reviewed. This  is  used  to  construct a  version  of  Bredon  cohomology  with  local  coefficients. In Section \ref{sectionktheory} the  construction  of  twisted equivariant  K-theory  for  proper and discrete actions  given  in  \cite{barcenasespinozajoachimuribe}  is  reviewed.
In  Section \ref{sectionspectral} , the  Bredon  cohomology  with local coefficients in twisted representations is shown  to  be isomorphic to the  second  page  of  a  spectral  sequence  converging  to  twisted  equivariant  K-theory. Some  phenomena concerning  the  third differential of this spectral sequence is also analyzed. In Section \ref{section_applications}  some simple examples are given including the case of discrete torsion
which was developed by Dwyer in \cite{Dwyer}.

\subsection*{Acknowledgements}

The  first  author    was financially  supported  by  the  Hausdorff  Center  for  Mathematics in Bonn, the  Leibniz Preis  of Prof. Wolfgang  L\"uck,   a  CONACyT postdoctoral fellowship , and PAPIIT-UNAM Grant ID100315. The  second author  acknowledges the  support  of a CONACyT postdoctoral  fellowship. The third author acknowledges and thanks for the financial support provided by the 
Alexander Von Humboldt Foundation, by the Max Planck Institute of Mathematics and by COLCIENCIAS through grant number FP44842-617-2014 of the Fondo Nacional de Financiamiento para la Ciencia, la Tecnolog\'ia y la Inovaci\'on. The  fourth  author  acknowledges the  support  of  an  UNAM postdoctoral  fellowship.  

\tableofcontents

\section{Bredon  cohomology  associated  to a  cover}\label{sectionbredon}

We  introduce  first  the  formalism  of modules  and  spaces over  a  category (see \cite{lueckdavis} for details). 

\begin{definition}
Let $\mathcal{C}$ be a small category. A  contravariant $\mathcal{C}$-space is a contravariant functor $\mathcal{C}\longrightarrow {\rm SPACES}$  to the category  of compactly generated spaces.
\end{definition}

\begin{definition}
Let $X$ and $Y$ be $\mathcal{C}$-spaces of  the  same  variance. Their mapping space ${\rm Hom}_{\mathcal{C}}(X,Y)$ is  the space of natural transformations between the functors $X$ and $Y$, endowed with the subspace topology of the product  of  the  spaces  of  pointed maps $\Pi_{c\in \mathrm{Obj}(\mathcal{C})} {\rm Map}(X(c),Y(c))$, where ${\rm Map}(X(c),Y(c))$ has the compact-open topology for any $c\in \mathrm{Obj}(\mathcal{C})$.

Let  $I$  be  the constant  functor  with  value  the  interval  $[0,1]$.   A  $\mathcal{C}$-homotopy  between  two  $\mathcal{C}$-maps  $f_{0}, f_{1} :X\to  Y$ is  a  natural  transformation  $H:X\times  I  \to  Y$  such  that  the  composition  $H\circ i_{k}$ with  the inclusions  $i_{k}: X\to X\times  I $  for  $k=0,1$ are equal to  $f_{k}$.  
The  set  of $\mathcal{C}$-homotopy  classes  of  maps   between  two  spaces  will be  denoted  by  $[X, Y]_{\mathcal{C}}$

\end{definition}

\begin{definition}
Let $X$ be a contravariant, pointed $\mathcal{C}$-space over $\mathcal{C}$ and  let $Y$ be a covariant $\mathcal{C}$-space over $\mathcal{C}$. Their tensor product $X\otimes_{\mathcal{C}} Y$ is  the  space  defined by
$$\coprod_{c\in \mathrm{Obj}(\mathcal{C)}}X(c)\times Y(c)/\sim$$
where $\sim$ is the equivalence relation generated by $(X(\phi)(x), y)\sim (x,Y(\phi)(y))$ for all morphisms $\phi:c\to d$ in $\mathcal{C}$ and points $x\in X(d)$ and $y\in Y(c)$.
\end{definition}

\begin{definition}
Let  $\mathcal{C}$  be  a  small  category. A  free $\mathcal{C}$-CW  complex is  a  contravariant $\mathcal{C}$-space  together  with a  filtration
$$ X_{0} \subset X_{1} \subset \ldots = X$$
 such  that  $X={\rm colim}_{n} X_{n}$  and  each  $X_{n}$  is  obtained  from the $X_{n-1}$  by  a  pushout consisting  of maps  of $\mathcal{C}$-spaces  of  the  form
$$\xymatrix{\coprod_{i\in I_{n}} {\rm Mor}_{\mathcal{C}}(?, c_{i})\times  S^{n-1} \ar[d]\ar[r] & X_{n-1}\ar[d]  \\ \coprod_{i\in I_{n}} {\rm Mor}_{\mathcal{C}}(?,c_{i} )\times  D^{n}\ar[r] &  X_{n}}$$
where the $c_i$'s are objects in ${\mathcal{C}}$ and  the  spaces ${\rm Mor}_{\mathcal{C}}(?, c_{i})$  carry the  discrete  topology. 
\end{definition}

\begin{definition}
Let $\mathcal{C}$ be  a  small  category  and $R$ be a  commutative  ring. A  contravariant  $R\mathcal{C}$-module  is  a  contravariant  functor  from  $\mathcal{C}$  to the  category  of  $R$-modules.  A contravariant  $R\mathcal{C}$-chain  complex  is  a functor  from $\mathcal{C}$  to the  category  of  $R$-chain  complexes. 
\end{definition}

We write $\calc$-module for a $\mathbb{Z}\calc$-module.

An  R$\mathcal{C}$-module $F$  is  free  if  it  is  isomorphic  to  an $R\mathcal{C}$-module  of  the  form $$F(?)= \bigoplus_{i\in I}R[{\rm Mor}_{\mathcal{C}}(?,c_{i})]$$
for  some  index set $I$ and  objects  $c_{i}\in \mathcal{C}$. 

Given  two  $R\mathcal{C}$-modules  $A$,  $B$ of the  same  variance, the $R$-module 
$${\rm Hom}_{R\mathcal{C}}(A,B)$$
is  the  module  of  natural  transformations  of functors from $\mathcal{C}$ to $R$-modules.

\begin{definition}
Given  a  category  $\calc$ and  an  object  $c$  in $\mathcal{C}$,  the  category  over $c$,  $ \calc \downarrow c$ is  the  category  where  the  objects are  morphisms $\varphi: c_{0}\to  c$  and  a  morphism   between $\varphi_{0}:   c_{0}\to  c$   and  $\varphi_{1}:  c_{1}\to  c $  is  a  morphism  $\psi:c_{0}\to  c_{1} $ in $\calc$  such  that  $\varphi_{0} = \varphi_{1}\circ \psi$.

Dually, the  category under  an  object  $c$, denoted  $c\downarrow \mathcal{C}$   is  the  category  where  the  objects  are  morphisms $\varphi: c\to c_{0}$  and  a  morphism  between  $\varphi_{0}:c\to  c_{0}$ and  $\varphi: c\to  c_{1}$    is  a  morphism  in $\mathcal{C}$,    $\psi:  c_{0}\to  c_{1}$  such  that  $\varphi_{1}= \psi \circ \varphi_{0}$.  

Fix    an  object  $c$, and   denote  by   $\B   \mathcal{C} \downarrow c$ the   classifying  space  of  the  category over  $c$ and  by  $\B c\downarrow \mathcal{C} $  the  classifying  space  of  the  category  under  $c$.   

The  contravariant, free  $\mathbb{Z}\calc$-chain  complex  $C_{*}^\mathbb{Z}(\calc)$   is   defined  on  every  object  as  the  cellular  $\mathbb{Z}$-chain  complex  of $\B \mathcal{C} \downarrow c$. 
\end{definition}

\begin{definition}
Let $M$  be  a  contravariant  $\mathcal{C}$-module.  The  cohomology  of  $\calc$ with  coefficients  in  $M$, $H^{n}(\calc, M)$   is  defined  to  be  the   cohomology  groups  of  the cochain  complex  of  natural  transformations  between  the $\calc$-modules $C_{*}^{\mathbb {Z}}(\calc) $  and  $M$,
$$H^{n}(\calc, M):=H^{n}{\rm Hom}_{\mathbb{Z}\calc}(C_{*}^{\mathbb {Z}}(\calc), M).$$ 
\end{definition}

We  now  specialize  to the  categories  relevant  to  twisted   K-theory  and   Bredon  cohomology  with local  coefficient  systems. 

Let $G$  be  a  group  and  $X$  be  a  proper  $G$-ANR. Let  $\mathcal{U}=  \{ U_i \}_{i\in \Sigma }$   be  a  countable covering  of  $X$    by   open,  $G$-invariant   sets $X= \underset{i \in \Sigma }\bigcup U_{i}$. Given a  subset  $\sigma\subset  \Sigma$,  define  $U_{\sigma}=\underset{i \in \sigma }{\cap} U_{i}$. 
We  will  assume  that   for  all $\sigma$, the open set  $U_{\sigma}$ is $G$-homotopy  equivalent  to an  orbit  $G/H_{\sigma}$  for  a  finite  group $H_{\sigma} \subset G$. The  existence  of  such  a  cover,  sometimes  known  as  \emph  {contractible  slice  cover},   is guaranteed  for  proper  $G$-ANR's  by an appropriate  version  of  the  slice  Theorem (see \cite{antonyan}).

The  category associated  to  $\mathcal{U}$ ,  denoted  by  $\cover$, has  for  objects  $ \underset{\sigma \subset \Sigma}{\bigsqcup} U_{\sigma  }$   and  for morphism  the inclusions $U_\tau \to U_\sigma$ whenever there is an inclusion of sets $\sigma \subset \tau$.

A  coefficient  system   with values  on  $R$-modules   is  a  contravariant    functor  $\cover\to  R-{\rm Mod}$. 
\begin{definition}
Let  $X$  be  a  proper  $G$-space with  a contractible slice  cover  $\mathcal{U}$, and  let  $M$  be  a  coefficient  system. Define  the Bredon cohomology  groups  with  respect  to $\mathcal{U}$  as  the  cohomology  groups   of  the category $\cover$ with coefficients in $M $,
$$H^{n}_G(X,\mathcal{U}; M):= H^{n}(\cover, M).$$
\end{definition}

Whenever we have a refinement $\mathcal{V} \to \mathcal{U}$ of the $G$-invariant cover, we get a group  homomorphism 
$$H^{n}_G(X,\mathcal{U}; M) \to H^{n}_G(X,\mathcal{V}; M')$$
where the functor $M'$ is obtained by the composition of the functor $\mathcal{N}_{G} \mathcal{V} \to \mathcal{N}_{G} \mathcal{U}$ with the functor $M$. 

\begin{remark}
For  more  general  spaces than a  proper $G$-ANR, a  version  of  \v{C}ech  cohomology  might  be  constructed by  taking  the  inverse  limit  over open covers  $\mathcal{U}$  of  the  space $X$:
$$\check{H}^{n}_G(X; M) := \lim_{\mathcal{U}} H^{n}(\cover, M).$$
Details  are  provided  in \cite{matumoto}. Other  approaches  to  \v{C}ech  versions  of Bredon  cohomology  include  \cite{honkasalo}. 
\end{remark}

\section{Parametrized equivariant topology}\label{sectionOGspaces}

The  orbit category   was  introduced  by Bredon   for  the  definition  of  cohomological  invariants  of  spaces  with  an  action. We introduce  now a  formalism  for  taking  into account  also  twisting  data. 


\begin{definition}
Let  $G$  be  a  discrete  group. The orbit category $\OO_G^P$, with respect to the family of finite subgroups, has as objects 
$${\rm{Obj}}(\OO_G^P) = \{G/H \mid H \ \mbox{is a finite subgroup of} \ G \}$$ and as morphisms $G$-maps
$${\rm Mor}_{\OO^P_G} (G/H,G/K) = {\rm Map}(G/H,G/K)^G.$$
\end{definition}

Given  a $G$-space $X$, the {\it fixed point set system} of $X$, denoted by $\Phi X$, is
the $\OO_G^P$-space defined by:
$$\Phi X (G/H) := {\rm{Map}}(G/H, X)^G= X^H$$
and if $\theta: G/H \to G/K$ corresponds to $gK \in (G/K)^H$ then
$$\Phi X (\theta)(x) := gx \in X^H$$
whenever $x \in X^K$. The functor $\Phi$ becomes a functor from
the category of proper $G$-spaces to the category of
$\OO_G^P$-spaces.

If $\XX$ is a contravariant functor from $\OO_G^P$ to spaces, and $\nabla$ is the  covariant functor from $\OO_G^P$   to spaces  which  assigns  to  an  orbit  $G/H$ the homogenous space $G/H$, one can define the $G$-space
$$\hat{\XX } := \bigsqcup_{c \in {\rm{Obj}}(\OO_G^P)} \XX(c) \times \nabla(c) / \sim$$
where $\sim$ is the equivalence relation generated by $(\XX(\phi)(x),y) \sim (x, \nabla(\phi)(y))$ for all morphisms $\phi: c \to d$ in $\OO_G^P$ and points $x \in \XX(d)$ and $y \in \nabla(c)$,  and  the  $G$-action  comes  from  the  left  translation  action on $G/H$.

For $G$ a discrete group \cite[lemma  7.2]{lueckdavis} the functors $\Phi$ and $ \_ \times_{\OO_G^P}\nabla$
are adjoint, i.e. for a $\OO_G^P$ space $\XX$ and a  $G$-space $Y$ there is a natural homeomorphism
$${\rm{Map}}(\XX \times_{\OO_G^P}\nabla,Y)^G \stackrel{\cong}{\to} {\rm{Hom}}_{\OO_G^P} (\XX, \Phi Y),$$
and moreover, the adjoint of the identity map on $\Phi Y$ under the above adjunction, is a  natural $G$-homeomorphism
$$(\Phi Y) \times_{\OO_G^P}\nabla \stackrel{\cong}{\to} Y.$$

A  model  for  the  homotopical version  of the previous construction is defined as follows.
 Consider   the  topological     category $(\XX, \nabla )$
whose objects are

$${\rm{Obj}}((\XX, \nabla ))=\bigsqcup_{c \in {\rm{Obj}}(\OO_G^P)} \XX(c) \times \nabla(c) $$
and whose morphisms consist of all triples $(x, \phi,y)$ where  $\phi: c \to d$ is a morphism in $\OO_G^P$ and
$x \in \XX(d)$ and $y \in  \nabla(c) $, with ${\rm source} (x, \phi,y)= (\XX(\phi)(x),y)$ and ${\rm {target}}(x, \phi,y)= (x, \nabla(\phi)(y))$. 
Define the space $\hat{\XX}^{h}$ as  the  geometric  realization  of  the  category  $(\XX, \nabla )$. The  space $\hat{\XX}^{h}$ is  provided  with  a  map  $\hat{\XX}^{h}\to  \hat{\XX}$ which is  a  model  for  the   map  from  the  homotopy  colimit  to  the  colimit. This  map   is  a  $G$-homotopy  equivalence   if  $\XX$  is  a  free  $\OO_G^P$-complex.

We  recall now results  on  the  homotopy theory  of   spaces  with an  action  of  a  group $G$  and  $\OO_G^P$-spaces.

\begin{definition}
Let $G$ be a discrete group. Given  a  family  $\mathcal{F}$ of  subgroups  of  $G$,  which  is  closed  under  conjugation  and taking  subgroups.  
\begin{itemize}
\item  A map $f \colon X \rightarrow Y$ of $G$-spaces is called an $\mathcal{F}$-equivalence if for every finite subgroup $H \leq G$, the map $f^H \colon X^H \rightarrow Y^H$ is a weak equivalence of topological spaces. 

\item A map $f \colon X \rightarrow Y$ of $G$-spaces is called an $\mathcal{F}$-fibration if for every finite subgroup $H \leq G$, the map $f^H \colon X^H \rightarrow Y^H$ is a Serre fibration of topological spaces.

\item A map $f \colon X \rightarrow Y$ of $G$-spaces is called an $\mathcal{F}$-cofibration if it has the left lifting property with respect to any map which is $\mathcal{F}$-equivalence and $\mathcal{F}$-fibration. 
\end{itemize}

\end{definition}

The  $qf$-model  structure  on $\OO_G^P$-spaces, with levelwise  weak  equivalences and   cofibrations  having  the  left homotopy extension property is  Quillen  equivalent  to the  homotopy category  of compactly generated, weak Hausdorff  $G$-spaces, with  the  above  mentioned  model  category  structure for  the  family $\mathcal{F}= \mathcal{ALL}$ of  all  subgroups of  $G$ \cite{elmendorf}.    In the Appendix we  prove  a  parametrized  version  of  this result and  we show some more facts concerning  the  homotopy  category  of  both  $G$-spaces  and $\OO_G^P$-spaces. 

We now introduce the category of $\OO_G^P$-spectra. Recall that a spectrum is a sequence of pointed spaces $\{ E_n\}_{n \in \IZ}$ with structure maps $E_{n}\wedge S^1\to E_{n+1}$. 

A (strong) map of spectra $f:E\to F$ is a sequence of maps compatible with the structure maps. 

Finally, recall that a spectrum is called an $\Omega$-spectrum if the adjoint of the structure maps $E_n \to \Omega E_{n+1}$ are weak homotopy equivalences. 

\begin{definition}
Let $G$ be a discrete group. An $\OO_G^P$-spectrum is a contravariant functor $E: \OO_G^P \to {\rm SPECTRA}$ to the category of spectra and strong maps. 
\end{definition}


Given an $\OO_G^P$-space $\XX$,  we denote by $$\Sigma \XX= \XX_+ \wedge S^1$$ the space given  on each object $G/H$ as the reduced suspension $\XX_{+}(G/H)\wedge S^1$, together with the structure maps given by smashing with the identity map.
 
The $n$-th suspension $\Sigma^n \XX$ is the space defined on objects as  $\XX(G/H)_+ \wedge S^n$.

\begin{definition}
Let $\XX$ be an $\OO_G^P$-space. The naive $\OO_G^P$-suspension spectrum of $\XX$, denoted by $\Sigma^\infty \XX $,  is defined on each  object $G/H$ as the $n$-th suspension space $\Sigma_{\OO_{G}^P}^\infty \XX(n) = \Sigma^n \XX= \XX_+ \wedge S^{n}$ with the $\OO_G^P$-structure maps obtained by smashing the $\OO_G^P$-maps of $\XX$ with the identity map $S^n\to S^n $ and spectra structure maps given by the homeomorphisms $S^n\wedge S^1 \to S^{n+1}$. 
\end{definition}

We now introduce parametrized versions of the  constructions defined in the orbit category. 

\begin{definition}\label{example OGoverB}
Fix a contravariant  $\OO_G^P$-space $\mathcal{B}$. A $\OO_G^P$-space over $\mathcal{B}$ is  a contravariant  $\OO_G^P$ space $\XX$ endowed with  a natural transformation of $\OO_G^P$-spaces $p_\XX: \XX\to \mathcal{B}$; this map is usually called projection. 

A map of $\OO_G^P$-spaces over $\mathcal{B}$ is a  map of $\OO_G^P$-spaces $F:\XX\to \YY$, which in addition is compatible with projections in the sense that $ p_{\YY}\circ F= p_{\XX}$. 
\end{definition}

The space of maps over $\mathcal{B}$, denoted by  ${\rm Hom}_{\OO_G^P}(\XX, \YY)_{\mathcal{B}}$, is defined as the subspace of the $\OO_G^P$-mapping space consisting of $\OO_G^P$-maps which are compatible with the projection  maps:
$$ {\rm Hom}_ {\OO_G^P}(\XX, \YY)_ {\mathcal{B}} :=\{F\in {\rm Hom}_{\OO_G^P}(\XX,\YY) \mid p_{\YY}\circ F= p_{\XX}\}.$$
We denote the set of homotopy classes  of maps over $\mathcal{B}$  by $$_{\OO_G^{P}}[\XX,\YY]_\mathcal{B} := \pi_0 \left( {\rm Hom}_ {\OO_G^P}(\XX, \YY)_ {\mathcal{B}} \right).$$

\section{Twisted equivariant K-theory and local coefficient versions of Bredon cohomology}\label{sectionktheory}

\subsection{Twisted equivariant K-theory}
Twisted  equivariant  K-theory  for  proper  actions  of  discrete  groups  was  introduced  in \cite{barcenasespinozajoachimuribe}. 
In what follows we will recall its definition using Fredholm bundles and its properties following \cite{barcenasespinozajoachimuribe}
and the classification of equivariant principal bundles done in \cite{barcenasespinozajoachimuribe, LueckUribe}.

\begin{definition} \label{def:proj_equi-bundle}
Let  $X$  be  a  proper  $G$-space with the  homotopy  type  of  a proper  $G$-ANR.  Let  $\mathcal{H}$  be  a separable complex  Hilbert  space and
$$\mathcal{U}(\mathcal{H})=\{U:\mathcal{H}\to  \mathcal{H}\mid  U\circ U^ {*}=U^ {*}\circ  U= {\rm Id}\}$$
be the unitary group endowed with the norm topology.
The group $P\calu(\calh)=\calu(\calh)/S^ {1}$ with the  quotient  topology  is the  group  of  projective  unitary  operators. 
    
A  projective  unitary, stable  $G$-equivariant  bundle  
is  a right  $P\calu(\calh)$,  principal  bundle  
$$P\calu(\calh) \to P \to X$$
 endowed with a left $G$
action lifting the action on $X$ such that:
\begin{itemize}
\item the left $G$-action commutes with the right
$P\calu(\calh)$ action, and \item for all $x \in X$ there exists a
$G$-neighborhood $V$ of $x$ and a $G_x$-contractible slice $U$ of
$x$ with $V$ equivariantly homeomorphic to $ U \times_{G_x} G$
with the action $$G_x \times (U \times G) \to U \times G, \ \ \ \
k \cdot(u,g)= (ku, g k^{-1}),$$
together with a local
trivialization
$$P|_V \cong  (P\calu(\calh) \times U) \times_{G_x} G$$ where the action of the isotropy group
is:
 \begin{eqnarray*}
 G_x \times [(P\calu(\calh) \times U) \times G] & \to & (P\calu(\calh) \times U) \times G
 \\
\ k \cdot [(F,y),g] & \mapsto & [(f_x(k)F, ky), g k^{-1}]
\end{eqnarray*} with $f_x : G_x \to P\calu(\calh)$ a fixed {\it{stable}}
homomorphism, in the  sense  that   the unitary representation $\calh$ induced by the
homomorphism $\widetilde{f_{x}}: \widetilde{G_x}= f_{x}^*\calu(\calh) \to
\calu(\calh)$ contains each of the irreducible representations of
$\widetilde{G}_x$ on  which  $S^1$ acts  by  multiplication  an infinite number of times. 
\end{itemize}
\end{definition}

Let $X$ be a $G$-space and $P \to X$ a projective unitary stable
$G$-equivariant bundle over $X$. Recall \cite{atiyahsegal, barcenasespinozajoachimuribe} that the space of Fredholm
operators  is endowed with a continuous right action
of the group $P\calu(\calh)$ by conjugation, therefore we can take
the associated bundle over $X$
$$\Fred(P) := P \times_{P\calu(\calh)} \Fred(\calh),$$
  where  $\Fred(\calh)$  is  the  space  of  Fredholm  operators  with the norm topology, and
 with the induced $G$ action given by
 $$g \cdot [(\lambda, A)] := [(g \lambda, A)]$$
for $g$ in $G$, $\lambda$ in $P$ and $A$ in $\Fred(\calh)$.

Denote by $$\Gamma(X; \Fred(P))$$ the space of sections of the
bundle $\Fred(P) \to X$ and choose as base point in this space the
section which chooses the identity operator on each fiber. This section
exists because the $P\calu(\calh)$-action on ${\rm{Id}}_\HH$
is trivial, and therefore $$X \cong P/P\calu(\calh) \cong P
\times_{P\calu(\calh)} \{{\rm{Id}}_\HH \} \subset \Fred(P);$$
let us denote this {\it{identity section}} by $s$.

The proof of Bott periodicity done in \cite[Theorem 5.1]{AtiyahSinger} shows the homotopy equivalence $\Omega^2(\Fred(\calh)) \simeq \Fred(\calh)$. This proof
can be carried without changes whenever a compact Lie
group $K$ acts in $\calh$ with infinitely many representations
for each irreducible representation appearing in $\calh$. Taking
equivariant Fredholm operators $\Fred(\calh)^K$ we obtain the homotopy
equivalence $\Omega^2(\Fred(\calh)^K) \simeq \Fred(\calh)^K$. Therefore we obtain Bott periodicity for
the twisted and equivariant case an we may define the twisted $G$-equivariant K-theory groups as follows.

\begin{definition} \label{definition K-theory of X,P}
  Let $X$ be a connected $G$-space and $P$ a projective unitary stable
  $G$-equivariant bundle over $X$. The {\it{twisted $G$-equivariant
  K-theory}} groups of $X$ twisted by $P$ are defined as the  homotopy  groups  of  the  $G$-equivariant  sections
 \begin{align*}
 K^{p}_G(X;P) &:= \pi_0 \left( \Gamma(X;\Fred(P))^G, s \right) & \mbox{whenever} \ \ p\ \ \mbox{is even}\\
 K^{p}_G(X;P) &:= \pi_1 \left( \Gamma(X;\Fred(P))^G, s \right) & \mbox{whenever} \ \ p\ \ \mbox{is odd}
 \end{align*}
  where $s$ denotes the identity section.
\end{definition}

\subsection{Universal projective unitary stable equivariant bundle}
In \cite[Section 3.2]{barcenasespinozajoachimuribe} it was constructed the universal projective unitary stable equivariant bundle
by gluing the universal bundles over each orbit type. Let us recall how this bundle is assembled since we need this information in order to define the Bredon cohomology with local coefficients.

 The base of this universal bundle was constructed from the   $\mathcal{O}_G^P$-space
$|{{\mathcal{C}}}|$ which at each orbit type $G/K$ assigns the space $|{{\mathcal{C}}}_{G/K}|$; this space is
the geometric realization of the groupoid $${{\mathcal{C}}}_{G/K}=[ \Funct_{st}(G \ltimes G/K, P \mathcal{U}(\mathcal{H})) / {\rm{Map}}(G/K, P \mathcal{U}(\mathcal{H}))]$$
 whose objects are functors $\Funct_{st}(G \ltimes G/K, P \mathcal{U}(\mathcal{H}))$
 from the category defined by the left $G$ action on $G/K$, denoted by $G \ltimes G/K$, and the category defined by the group 
$P \mathcal{U}(\mathcal{H})$
whose restriction to ${\mathrm{Hom}}(K,P \mathcal{U}(\mathcal{H}))$ are stable homomorphisms,
and whose morphisms are given by natural transformations ${\rm{Map}}(G/K, P \mathcal{U}(\mathcal{H}))$.

 In the category of $\mathcal{O}_G^P$-spaces, a classifying map for the bundle $\Phi P \to \Phi X$  is obtained by  map $\mu: \Phi X \to |{{\mathcal{C}}}|$ assembling the maps
$\mu_{G/K}: X^K \to |{{\mathcal{C}}}_{G/K}|$, with the property that 
$$ (\mu_{G/K})^* |{{\mathcal{D}}}_{G/K}| \cong P|_{X^K}$$
where $|{{\mathcal{D}}}_{G/K}| \to |{{\mathcal{C}}}_{G/K}|$ is the universal projective unitary stable $N_G(K)$-equivariant bundle
over $|{{\mathcal{C}}}_{G/K}|$ and which is defined as follows \cite[Def. 4.1]{barcenasespinozajoachimuribe}: the morphisms  $\mathrm{Mor}({{\mathcal{D}}}_{G/K})$ are 
$$\Funct_{st}(G \ltimes G/K, P \mathcal{U}(\mathcal{H})) \times  P \mathcal{U}(\mathcal{H}) \times  {\rm{Map}}(G/K, P \mathcal{U}(\mathcal{H}))$$
and the objects $\mathrm{Obj}({{\mathcal{D}}}_{G/K})$
$$\Funct_{st}(G \ltimes G/K, P \mathcal{U}(\mathcal{H})) \times  P \mathcal{U}(\mathcal{H}), $$
with structural maps $\mathrm{source}(\psi,F,\sigma)=(\psi,F)$, $\mathrm{target}(\psi,F, \sigma)=( \sigma F^{-1}\psi F \sigma^{-1}, \sigma([K]))$ and composition  $\mathrm{comp}((\psi,F,\sigma),( \sigma F^{-1}\psi F \sigma^{-1}, \sigma([K]), \delta)= ( \psi,F,\delta\sigma([K])^{-1}\sigma)$. The fucntor ${{\mathcal{D}}}_{G/K} \to {{\mathcal{C}}}_{G/K}$
forgets the $P \mathcal{U}(\mathcal{H})$ component, and the map $|{{\mathcal{D}}}_{G/K}| \to |{{\mathcal{C}}}_{G/K}|$
denotes the map of the geometric realizations.

Denote by $\Fred(|{{\mathcal{D}}}|))$ the $\mathcal{O}_G^P$-space over $|{{\mathcal{C}}}|$ defined on the orbit type
$G/K$ by
$$\Fred(|{{\mathcal{D}}}_{G/K}|))^K:= \left(|{{\mathcal{D}}}_{G/K}| \times_{P \mathcal{U}(\mathcal{H})} \Fred(\mathcal H)\right)^K$$
and denote by $p:\Fred(|{{\mathcal{D}}}|) \to |{{\mathcal{C}}}|$ the projection map which is the assembly of the canonical projection maps 
$$p_{G/K} : \Fred(|{{\mathcal{D}}}_{G/K}|)^K \to |{{\mathcal{C}}}_{G/K}|.$$

Since the identity operator ${\rm{Id}}_{\mathcal H}$ on the Hilbert space $\mathcal H$ is invariant under the conjugation action of $P \mathcal{U}(\mathcal{H})$, then the projection map $p$ has a canonical section $$s:  |{{\mathcal{C}}}| \to \Fred(|{{\mathcal{D}}}|)$$
which assigns to every point the operator ${\rm{Id}}_{\mathcal H}$.

Alternatively, we could define the twisted equivariant K-theory groups in the category of $\mathcal{O}_G^P$-spaces in the following way. For a proper $G$-CW complex $X$  endowed with a map of $\mathcal{O}_G^P$-spaces $\mu: \Phi X \to |{{\mathcal{C}}}|$, we can alternatively define the twisted equivariant K-theory groups of the pair $(\Phi X;\mu)$ as
the homotopy groups of the pointed space
$$\left({\rm Hom}_{\OO_G^ P}(\Phi X, \Fred(|{{\mathcal{D}}}|))_ {|{{\mathcal{C}}}|}, s \circ \mu \right)$$
namely
\begin{align*}
K_G^{p}(\Phi X;\mu) &:=\pi_0   \left({\rm Hom}_{\OO_G^ P}(\Phi X, \Fred(|{{\mathcal{D}}}|))_ {|{{\mathcal{C}}}|}, s \circ \mu \right) && \mbox{whenever} \ \ p \ \ \mbox{is even}\\
K_G^{p}(\Phi X;\mu) &:=\pi_1   \left({\rm Hom}_{\OO_G^ P}(\Phi X, \Fred(|{{\mathcal{D}}}|))_ {|{{\mathcal{C}}}|}, s \circ \mu \right)&& \mbox{whenever} \ \ p \ \ \mbox{is odd.}
\end{align*}


\begin{remark}We would like to note here that an alternative, and homotopically equivalent, construction of the universal projective unitary stable equivariant bundle was done in \cite[Section 15]{LueckUribe}.
There all topological issues regarding the existence of local sections were resolved. 
\end{remark}

\subsection{Bredon cohomology with  local coefficients}\label{section bredon homotopic}

The local coefficients for the Bredon cohomology that we are
going to define in this section are constructed from the 
fiberwise homotopy groups of the fiber bundles
$$p_{G/K} : \Fred(|{{\mathcal{D}}}_{G/K}|)^K \to |{{\mathcal{C}}}_{G/K}|.$$
In order to have an explicit definition of these local coefficients, we need to recall some properties of the previous fibration. 

The only non trivial homotopy groups of the spaces $|{{\mathcal{C}}}_{G/K}|$ exist in degree 0, 1 and 3.
We know by \cite[Theorem 1.9]{barcenasespinozajoachimuribe} that $\pi_0(|{{\mathcal{C}}}_{G/K}|) \cong {\rm{Ext}}(K,S^1)$
is the set of isomorphism classes of $S^1$-central extensions of $K$, and that the fundamental group of each connected component of 
$|{{\mathcal{C}}}_{G/K}|$ is isomorphic to ${\rm{Hom}}(K,S^1)$. Let us denote by $|{{\mathcal{C}}}_{G/K}|_{\widetilde{K}}$ the connected components of $|{{\mathcal{C}}}_{G/K}|$  associated the $S^1$-central extension $\widetilde{K}$; hence
\begin{align*}
{{|{{\mathcal{C}}}_{G/K}|}} & = \bigcup_{\widetilde{K} \in {\rm{Ext}}(S^1,K)}|{{\mathcal{C}}}_{G/K}|_{\widetilde{K}}.
\end{align*}

Now, for any point $ x \in |\mathcal{C}_{G/K}|_{\widetilde{K}}$ there is associated a specific stable homomorphism $\alpha : K \to  P \mathcal{U}(\mathcal{H})$ with $\widetilde{K} \cong \alpha^*(\mathcal{U}(\mathcal{H}))$ and lift $\widetilde{\alpha}: \widetilde{K} \to  \mathcal{U}(\mathcal{H})$ such that the fiber $$p_{G/K}^{-1}(x) \subset \left(|{{\mathcal{D}}}_{G/K}| \times_{P \mathcal{U}(\mathcal{H})} \Fred(\mathcal H)\right)^K$$
is isomorphic to the  space of $K$-invariant Fredholm operators
$$\Fred(\mathcal H)^{\widetilde{\alpha}} := \{ F \in \Fred(\mathcal H) \mid \widetilde{\alpha}(k) F = F\widetilde{\alpha}(k)  \text{ for all } k \in \widetilde{K}\}. $$
The index map
\begin{eqnarray*}
\mathrm{ind} : \Fred(\mathcal H)^{\widetilde\alpha} &\to& R_{S^1}(\widetilde{K})\\
F &\mapsto& [\ker(F)] - [\coker(F)]
\end{eqnarray*}
is a homomorphism of groups that induces an isomorphism of groups at the level of the connected components 
\begin{align} \label{iso-index-on-fiber}
\mathrm{ind} : \pi_0(p_{G/K}^{-1}(x)) =\pi_0(\Fred(\mathcal H)^{\widetilde{\alpha}}) \stackrel{\cong}{\to} R_{S^1}(\widetilde{K});\end{align}
here $R_{S^1}(\widetilde{K})$ denotes the Grothendieck group of isomorphism classes of $\widetilde{K}$ representations where $\ker( \widetilde{K} \to K)$ acts by multiplication of scalars.
Hence we have that the connected components of the fibers of the map
$$p_{G/K} : \Fred(|{{\mathcal{D}}}_{G/K}|)^K|_{|{{\mathcal{C}}}_{G/K}|_{\widetilde{K}}} \to |{{\mathcal{C}}}_{G/K}|_{\widetilde{K}}$$
are all isomorphic to the group $R_{S^1}(\widetilde{K})$ via the {\it{index}} map. 
\begin{definition} \label{def:space_W_of_reps}
Consider the  $\OO_G^ P$-space ${\mathfrak{TR}}_0$ over $|{\mathcal{C}} |$ which at each orbit type $G/K$ is defined by
$$
({\mathfrak{TR}}_0)_{G/K} := \bigsqcup_{\widetilde{K} \in {\rm{Ext}}(K,S^1)} \overline{|{{\mathcal{C}}}_{G/K}|_{\widetilde{K}}} \times_{{\rm{Hom}}(K,S^1)} R_{S^1}(\widetilde{K})$$
where $\overline{|{{\mathcal{C}}}_{G/K}|_{\widetilde{K}}}$ is the universal cover of ${|{{\mathcal{C}}}_{G/K}|_{\widetilde{K}}}$, the action of ${\rm{Hom}}(K,S^1)$ on the left hand side is given by an explicit isomorphism $\pi_1(|{{\mathcal{C}}}_{G/K}|_{\widetilde{K}}) \cong{\rm{Hom}}(K,S^1) $ and the action on the right hand side is given by 
$${\rm{Hom}}(K,S^1) \times R_{S^1}(\widetilde{K}) \to R_{S^1}(\widetilde{K}),\hspace{0.5cm} (\rho , V) \mapsto \overline{\rho} \otimes_{\mathbb{C}} V $$
where $\overline{\rho}$ is understood as the 1-dimensional representation of $\widetilde{K}$ that the homomorphism $\rho$ defines.
Denote  by $t : |\mathcal{C}_{G/K}| \to ({\mathfrak{TR}}_0)_{G/K}$ the $0$-section.
\end{definition}

Note that the definition of the explicit isomorphism $\pi_1(|{{\mathcal{C}}}_{G/K}|_{\widetilde{K}}) \cong{\rm{Hom}}(K,S^1) $ is based on the following construction. The first two homotopy groups of $|{{\mathcal{C}}}_{G/K}|$ come from the first two 
homotopy groups of $\mathrm{Hom}_{st}(K,  P \mathcal{U}(\mathcal{H}))$, the space of stable homomorphisms.
Denote by $\mathrm{Hom}_{st}(K,  P \mathcal{U}(\mathcal{H}))_{\widetilde{K}}$ the connected component 
that defines $\widetilde{K}$ and let
$\mathrm{Hom}_{S^1}(\widetilde{K}, \mathcal{U}(\mathcal{H}))$
be the space of homomorphisms such that $\ker(\widetilde{K} \to K)$  
acts on $\mathcal{H}$ by multiplication.
Then the projection map
$$ \mathrm{Hom}_{S^1}(\widetilde{K}, \mathcal{U}(\mathcal{H})) \to 
\mathrm{Hom}_{st}(K,  P \mathcal{U}(\mathcal{H}))_{\widetilde{K}}$$
is a principal $\mathrm{Hom}(K,S^1)$-bundle where $\mathrm{Hom}(K,S^1)$
acts on $ \mathrm{Hom}_{S^1}(\widetilde{K}, \mathcal{U}(\mathcal{H}))$  by multiplication \cite[Prop. 15.7]{LueckUribe}, and therefore the projection
map is a universal cover for the base.

For a stable homomorphism $\alpha : K \to P\mathcal{U}(\mathcal{H})$ such that $\alpha^* \mathcal{U}(\mathcal{H}) \cong \widetilde{K}$, we choose a lift
$\widetilde{\alpha} : \widetilde{K} \to \mathcal{U}(\mathcal{H})$ in order to define the index map
\begin{eqnarray*}
\mathrm{ind}^{\widetilde\alpha} : \Fred(\mathcal H)^{\alpha} &\to& R_{S^1}(\widetilde{K})\\
F &\mapsto& [\ker(F)] - [\coker(F)]
\end{eqnarray*}
Whenever we choose another lift $\widetilde{\alpha}'=\widetilde{\alpha} \cdot \rho$ with $\rho: K \to S^1$, we have that
$\mathrm{ind}^{\widetilde{\alpha}'}(F)= \mathrm{ind}^{\widetilde{\alpha}}(F) \cdot \overline{\rho}$, and since the structural group 
of $\Fred(|\mathcal{D}_{G/K}|)^K$ is connected, we have that the fiberwise index map
\begin{align}  \label{diagram:Fred^K_W_TR_0}\xymatrix{
\Fred(|\mathcal{D}_{G/K}|)^K \ar[r]^-{\mathrm{ind}} \ar[d]_{p_{G/K}} & ({\mathfrak{TR}}_0)_{G/K} \ar[d]_{q_{G/K}} \\
 |\mathcal{C}_{G/K}|  \ar@/_/[u]_{s_{G/K}} \ar[r]_{=} &  |{\mathcal{C}_{G/K}| } \ar@/_/[u]_{t_{G/K}}}
\end{align}
is a well defined map of fiber bundles, and that it induces an isomorphism of the connected components of the fibers
\begin{align} \label{fiber_iso_Fred_W}
\pi_0(p_{G/K}^{-1}(x)) \cong R_{S^1}(\widetilde{K})
\end{align}
for every point $x \in |\mathcal{C}_{G/K}|_{\widetilde{K}}$ and every $S^1$-central extension $\widetilde{K}$. Assembling these maps we obtain an index map at the level of the $\OO_G^ P$-spaces over $|{\mathcal{C}} |$
\begin{align} \label{diagram:Fred_W_TR}
\xymatrix{
\Fred(|\mathcal{D}|) \ar[r]^-{\mathrm{ind}} \ar[d]_{p} & {\mathfrak{TR}}_0 \ar[d]_{q} \\
 |\mathcal{C}|  \ar@/_/[u]_{s} \ar[r]_{=} &  |{\mathcal{C}| } \ar@/_/[u]_{t}}
\end{align}
which induces an isomorphism on the connected components
of the fibers.

To construct the Bredon cohomology with coefficients in twisted representations, we perform a construction similar to the one done in Definition \ref{def:space_W_of_reps} but we replace the group of twisted representations $R_{S^1}(\widetilde{K})$ by $HR_{S^1}(\widetilde{K})$,  the Eilenberg-MacLane spectrum of the abelian group $R_{S^1}(\widetilde{K})$.

Denote by $ HR_{S^1}(\widetilde{K})$ the Eilenberg-MacLane
spectrum associated to the group $R_{S^1}(\widetilde{K})$,
i.e. at level $n \geq 0$ we have $ (HR_{S^1}(\widetilde{K}))_n=K(R_{S^1}(\widetilde{K}),n)$ where
$K(R_{S^1}(\widetilde{K}),n)$ is a  functorial model for the Eilenberg-MacLane space whose only non-trivial homotopy group is $R_{S^1}(\widetilde{K})$ in degree $n$, and
which comes endowed with weak homotopy equivalences
$\Omega K(R_{S^1}(\widetilde{K}),n) \simeq K(R_{S^1}(\widetilde{K}),n+1)$.

\begin{definition} \label{def:space_TR_of_reps}
For $n \geq 0$ consider the  $\OO_G^ P$-space ${\mathfrak{TR}}_n$ over $|{\mathcal{C}} |$ of twisted representations, such that on the orbit type $G/K$ we have
$$({\mathfrak{TR}}_n)_{G/K} := \bigsqcup_{\widetilde{K} \in {\rm{Ext}}(K,S^1)} \overline{|{{\mathcal{C}}}_{G/K}|_{\widetilde{K}}} \times_{{\rm{Hom}}(K,S^1)} K(R_{S^1}(\widetilde{K}),n)$$
where $\overline{|{{\mathcal{C}}}_{G/K}|_{\widetilde{K}}}$ is the universal cover of ${|{{\mathcal{C}}}_{G/K}|_{\widetilde{K}}}$, the action of ${\rm{Hom}}(K,S^1)$ on the left hand side is given by an explicit isomorphism $\pi_1(|{{\mathcal{C}}}_{G/K}|_{\widetilde{K}}) \cong{\rm{Hom}}(K,S^1) $ and the action on the right hand side is the one induced on the Eilenberg-MacLane space $K(R_{S^1}(\widetilde{K}),n)$ by the action
$${\rm{Hom}}(K,S^1) \times R_{S^1}(\widetilde{K}) \to R_{S^1}(\widetilde{K}),\hspace{0.5cm} (\rho , V) \mapsto \overline{\rho} \otimes_{\mathbb{C}} V. $$
Denote by $r_n: {\mathfrak{TR}}_n \to |\mathcal{C}|$ the natural projection map and by $\sigma_n:  |\mathcal{C}| \to 
{\mathfrak{TR}}_n$  the section which chooses the base point in $K(R_{S^1}(\widetilde{K}),n)$.
For $n<0$ let $ {\mathfrak{TR}}_n:=|{\mathcal{C}} |$ with
$r_n=\sigma_n= \mathrm{Id}_{|{\mathcal{C}} |}$.
\end{definition}

The weak homotopy equivalences $\Omega K(R_{S^1}(\widetilde{K}),n) \simeq K(R_{S^1}(\widetilde{K}),n+1)$
induce weak homotopy equivalences $\Omega {\mathfrak{TR}}_n \simeq {\mathfrak{TR}}_{n+1}$ in the category of $\OO_G^ P$-spaces over $|{\mathcal{C}} |$.
Assembling the spaces  ${\mathfrak{TR}}= \{ {\mathfrak{TR}}_n\}_{n \in \IZ}$ we obtain the following lemma.
\begin{lemma}
 ${\mathfrak{TR}}= \{ {\mathfrak{TR}}_n\}_{n \in \IZ}$ is a $\Omega$-spectrum in the category of
 $\OO_G^ P$-spaces over $|{\mathcal{C}} |$.
\end{lemma}

We are now ready to define the twisted Bredon cohomology associated to twisted representations.

\begin{definition} Let $X$ be a proper $G$-ANR endowed with a fixed map of $\mathcal{O}_G^P$-spaces $\xi: \Phi X \to |\mathcal{C}|$. The Bredon cohomology groups with local  coefficients in twisted representations associated to the pair $(\Phi X;\xi)$ are defined as the connected components of the based spaces
$ {\rm{Hom}}_{\mathcal{O}_G^P} (\Phi X , {\mathfrak{TR}}_n)_{|\mathcal{C}|}$, i.e.
$$\mathbb{H}^{p}_G(\Phi X,\xi) := \pi_0 \left(  {\rm{Hom}}_{\mathcal{O}_G^P} (\Phi X , {\mathfrak{TR}}_p)_{|\mathcal{C}|}; \sigma_p \circ \xi \right).$$
 \end{definition}

Alternatively, in the category of $\mathcal{O}_G^P$-spectra over 
$|\mathcal{C}| $ we have:
$$\mathbb{H}^{p}_G(\Phi X,\xi) := \pi_p \left(  {\rm{Hom}}_{\mathcal{O}_G^P} (\Sigma^\infty \Phi X , {\mathfrak{TR}})_{|\mathcal{C}|}; \sigma \circ \xi \right).$$

These cohomology groups satisfy the axioms of a parametrized $G$-equivariant cohomology theory and the proof follows the same lines
as the one for the twisted equivariant K-theory groups which can be found in \cite[Chapter 5]{barcenasespinozajoachimuribe};  we will not reproduce its proof here.

\begin{remark}Other approaches  to Bredon  cohomology  with  local  coefficients  include  \cite{basusen},  where methods  from  the theory of  crossed  complexes  and  their  classifying  spaces are  used to produce a  classifying  object  for  Bredon  cohomology  with local  coefficients.\end{remark}

\section{Segal's spectral sequence for twisted equivariant K-theory}\label{sectionspectral}
We  will  use  Segal's method   \cite{Segal}  to  obtain   a  filtration  of  the  homotopy  theoretically  defined  twisted  equivariant  K-theory,  as  well  as  a version  of  Bredon cohomology  associated  to  a  cover  to  handle   the   homotopical version  of  Bredon cohomology  described  in  the  previous section. We describe first the local coefficient system associated to twisted equivariant K-theory.

\begin{definition}[Local  coefficient  system of  twisted  equivariant K-theory]
Consider  a projective  unitary  stable  bundle  $P$ over a  proper  $G$-space  $X$ and a $G$-invariant and countable cover $\mathcal{U}$ for which each open set $U_\sigma$ is equivariantly contractible, i.e. $G$-homotopic to $G/H_\sigma$ for some  finite subgroup $H_\sigma$ depending on the set $U_\sigma$.  We can define   local coefficient  systems by the functors
\begin{align*}
\mathcal{K}^p_G( ?,P_?):\cover &\to \mathbb{Z}- {\rm  Mod}\\
U_\tau \subset U_\sigma & \mapsto K^{p}_G(U_\sigma;P|_{U_\sigma}) \to  K^{p}_G(U_\tau;P|_{U_\tau})
\end{align*}
 \end{definition}

\begin{proposition} \label{proposition spectral sequence  twisted}
Let  $X$  be  a  proper compact $G$-ANR and $P$ a projective  unitary  stable equivariant  bundle. Then Segal's spectral sequence applied 
to $\ktheory{G}{*}{X}$ and associated to the locally finite and equivariantly contractible cover $\mathcal{U}$, has as second page $E_2^{p,q}$ the cohomology of $\cover$
with coefficients in the functor $\mathcal{K}^0_G( ?,P|_?)$ whenever $q$ is even, i.e.
\begin{equation}\label{spseq}
E_{2}^{p,q}:= 
H^{p}_G(X,\mathcal{U}; \mathcal{K}^0_G( ?,P|_{?})) 
\end{equation}
and is trivial if $q$ is odd.
Its higher differentials
$$d_{r}:E_{r}^{p, q}\to  E_{r}^{p+r, q-r+1}$$
vanish  for    $r$ even.

\end{proposition}

\begin{proof}
Since the cover consists of equivariantly contractible spaces we know that the groups $K_G^q(U_\sigma; P|_{U_\sigma})$ are periodic
and trivial for $q$ odd. Therefore the fact that the second page of Segal's spectral sequence is isomorphic to $H^{p}_G(X,\mathcal{U}; \mathcal{K}^q_G( ?,P|_{?}))$
follows directly from Segal's original proof. Bott's isomorphism
implies that $K_G^{2n}(U_\sigma; P|_{U_\sigma}) \cong K_G^0(U_\sigma; P|_{U_\sigma})$ and therefore we have  that the even differentials vanish. \end{proof}

We are mainly interested in the second page of the spectral sequence, and to understand it we need to elaborate on the cohomology of $\cover$ with coefficients in the  functor $\mathcal{K}^0_G( ?,P|_{?}))$ and  compare it with  the  homotopy  theoretic definition  given in  Section \ref{sectionktheory}. This comparison will
be done in the category of $\mathcal{O}_G^P$-spaces over $|\mathcal{C}|$.

\vspace{0.5cm}

We claim the following result:

\begin{theorem} \label{theorem second page is Bredon}
Let $\mathcal{U}$ be a locally finite  cover of $G$-invariant sets of $X$ such that each non-trivial intersection of sets in the cover is equivariantly contractible. Then for any map $\mu: \Phi X \to |{{\mathcal{C}}}|$ the second page of Segal's spectral sequence applied
to the groups
$K_G^*(\Phi X;\mu)$ is isomorphic to the Bredon cohomology groups with local coefficients in twisted representations 
$\mathbb{H}^{p}_G(\Phi X,\mu)$;  i.e.
for $q$ even
$$E_2^{p,q} \cong \mathbb{H}^{p}_G(\Phi X,\mu).$$
\end{theorem}

    \begin{proof}
   Applying Segal's spectral sequence to $ \mathbb{H}^{p}_G(\Phi X,\mu)$ with the cover $\mathcal{U}$ we get
   that the second page of this spectral sequence is   
   $$\bar{E}_2^{p,q} = H_G^p(\Phi X,\mathcal{U}; \mathbb{H}^q_G(?, \mu|_?)).$$
   Since the open sets $U_\sigma$ are equivariantly contractible we have that for $q\neq 0$
   $$\mathbb{H}^q_G(\Phi U_\sigma, \mu|_{\Phi U_\sigma})) = 0   $$
and therefore $\bar{E}_2^{p,q}=0$ for $q\neq 0$. Therefore the spectral sequence collapses at the second page
   and this page becomes
   
   $$\bar{E}_2^{p,0} = H_G^p(\Phi X,\mathcal{U}; \mathbb{H}^0_G(?, \mu|_?)) \cong    \mathbb{H}^p_G(\Phi X, \mu)    $$
   where $\mathbb{H}^0_G(?, \mu|_?)$ is the local coefficient system defined by \begin{align*}
\mathbb{H}^0_G(?, \mu|_?):\cover &\to \mathbb{Z}- {\rm  Mod}\\
U_\tau \subset U_\sigma & \mapsto \mathbb{H}^0_G(\Phi U_\sigma, \mu|_{\Phi U_\sigma}) \to  \mathbb{H}^0_G(\Phi U_\tau,\mu|_{\Phi U_\tau}).
\end{align*}

Now we need to show that there is a canonical way to assign isomorphisms
  \begin{align} \label{local iso: twisted K with Bredon}
  \phi_\sigma : K_G^0(\Phi U_\sigma;\mu|_{\Phi U_\sigma}) \stackrel{\cong}{\to} \mathbb{H}^0_G(\Phi U_\sigma,\ \mu|_{\Phi U_\sigma})
  \end{align}
  which commute  with the restriction maps on each side; the existence of such isomorphisms would induce a canonical isomorphism
  between the complexes defined in the first page of the spectral sequences
  $$E_1^{p,0} \stackrel{\cong}{\to} \bar{E}_1^{p,0}$$
   and therefore would induce an isomorphism at the second pages
 $$E_2^{p,0} \stackrel{\cong}{\to} \bar{E}_2^{p,0}.$$
  The existence of the isomorphisms described in \eqref{local iso: twisted K with Bredon} and which are compatible with the inclusions, follow from the explicit maps described in diagrams \eqref{diagram:Fred^K_W_TR_0}, \eqref{diagram:Fred_W_TR} and from the isomorphisms of equations \eqref{fiber_iso_Fred_W}.
  Since $U_\sigma$ is equivariantly contractible to a point, by equation \eqref{iso-index-on-fiber} we know that the index map
  \begin{align*}
  {\rm Hom}_{\OO_G^ P}(\Phi U_\sigma, \Fred(|{{\mathcal{D}}}|))_{|{{\mathcal{C}}}|}  &\to
  {\rm Hom}_{\OO_G^ P}(\Phi U_\sigma, {\mathfrak{TR}}_0)_{|{{\mathcal{C}}}|}\\
  f & \mapsto \mathrm{ind} \circ f
  \end{align*}
  induces an isomorphism on connected components, and hence a canonical isomorphism
  \begin{align*}
 \phi_\sigma:K_G^0(\Phi U_\sigma;\mu|_{\Phi U_\sigma})  \stackrel{\cong}{\to}
  \mathbb{H}^0_G(\Phi U_\sigma;\mu|_{\Phi U_\sigma}) .
  \end{align*}
 The inclusion $U_\tau \subset U_\sigma$  induces a commutative diagram
 $$\xymatrix{
  {\rm Hom}_{\OO_G^ P}(\Phi U_\sigma, \Fred(|{{\mathcal{D}}}|))_{|{{\mathcal{C}}}|} \ar[d]\ar[r]^-{\mathrm{ind}} &
  {\rm Hom}_{\OO_G^ P}(\Phi U_\sigma, {\mathfrak{TR}}_0)_{|{{\mathcal{C}}}|} \ar[d]\\
  {\rm Hom}_{\OO_G^ P}(\Phi U_\sigma, \Fred(|{{\mathcal{D}}}|))_{|{{\mathcal{C}}}|}\ar[r]^-{\mathrm{ind}} &
  {\rm Hom}_{\OO_G^ P}(\Phi U_\sigma, {\mathfrak{TR}}_0)_{|{{\mathcal{C}}}|}
 }$$
 which implies that the isomorphisms $\phi_\sigma$ are compatible with restrictions, i.e. we have the commutative diagram
 $$\xymatrix{
  K_G^0(\Phi U_\sigma;\mu|_{\Phi U_\sigma})  \ar[r]^{\phi_\sigma}_\cong \ar[d] &
  \mathbb{H}^0_G(\Phi U_\sigma;\mu|_{\Phi U_\sigma}) 
  \ar[d]\\
  K_G^0(\Phi U_\tau;\mu|_{\Phi U_\tau})  \ar[r]^{\phi_\tau}_\cong  &
  \mathbb{H}^0_G(\Phi U_\tau;\mu|_{\Phi U_\tau}) .
 }$$
 The isomorphisms $\phi_\sigma$ induce the desired isomorphism $E_1^{p,0} \stackrel{\cong}{\to} \bar{E}_1^{p,0}$, and since they are compatible with restrictions, they induced an isomorphism of complexes
 thus preserving the first differential. This implies that 
 $E_2^{p,0} \stackrel{\cong}{\to} \bar{E}_2^{p,0}.$
 Bott periodicity implies that there
 are canonical isomorphisms $E_1^{p,q}\cong E_1^{p,0}$
 for $q$ even,  which are compatible with the restrictions. 
 We conclude that 
 $$E_2^{p,q} \cong \mathbb{H}^p_G(\Phi X;\mu)$$
 whenever $q$ is even and $E_2^{p,q}=0$ whenever $q$ is odd.
  \end{proof}

   \subsection{The third differential}
   The third differential on Segal's spectral sequence 
   $$d_3 : E_2^{p,q} \to E_{2}^{p+3,q-2},$$
   together with the isomorphism of Theorem \ref{theorem second page is Bredon}, 
   induce a degree three map 
   $$d_3 : \mathbb{H}^p_G(\Phi X, \mu) \to \mathbb{H}^{p+3}_G(\Phi X, \mu)$$
   on the Bredon cohomology with local coefficients in twisted representations which we will denote with the same symbol $d_3$.
   
  The purpose of this section is to evidence some  particular phenomena  concerning   this  differential $d_3$. 
  
\subsubsection{$G$-Invariant cohomology class}
  
Consider the trivial subgroup $\{1\} \subset G$ and recall that the bundle $| {{\mathcal{D}}}_{G/\{1\}}| \to | {{\mathcal{C}}}_{G/\{1\}}|$ is a universal projective unitary bundle thus having that $| {{\mathcal{C}}}_{G/\{1\}}|$ is a $K(\IZ,3)$. Hence 
for any map $\mu : \Phi X \to |{{\mathcal{C}}}|$ which classifies a projective equivariant stable unitary bundle over $X$, the map $\mu_{G/\{1\}}: X \to | {{\mathcal{C}}}_{G/\{1\}}|$ encodes the information of the projective unitary bundle once the $G$-action is forgotten.
The map  $\mu_{G/\{1\}}$ defines a degree 3 cohomology
class $\eta \in H^3(X, \IZ)$ which is moreover $G$-invariant.

    In cohomological terms we know that the bundle $P \to X$ is classified by an element $\overline{\eta} \in H^3(X \times_G EG , \mathbb{Z})$.
    Denoting by $\eta$ the restriction of $\overline{\eta}$ to any fiber of the Serre fibration $X \to X \times_G EG \to BG$, and restricting it further
    to the fixed point set of the group $K$, we get a class
    $$\eta_K:= \eta|_{X^K} \in H^3(X^K;\mathbb{Z}).$$
    This class $\eta_{K}$ is precisely the class defined by the  the composition $X^K \stackrel{\mu_{G/K}}{\to} |{{\mathcal{C}}}_{G/K}| \stackrel{\kappa_{G/K}}{\to} |{{\mathcal{C}}}_{G/\{1\}}| $, and it is 
    furthermore $N_G(K)/K$-invariant. 
    
    Since the groups $R_{S^1}(\widetilde{K})$ are free $\IZ$-modules, there is an induced structure at the level of the Eilenberg-MacLane spaces
    $$| {{\mathcal{C}}}_{G/\{1\}}|\times K(R_{S^1}(\widetilde{K}),n) \to K(R_{S^1}(\widetilde{K}),n+3)$$
    which is $N_G(K)/K$-equivariant, compatible with restrictions  and which recovers the cup product by a degree 3 cohomology class. 
    Composing with the canonical maps $\kappa_{G/K}: |{{\mathcal{C}}}_{G/K}| \to |{{\mathcal{C}}}_{G/\{1\}}|$ we obtain maps
   $$\varepsilon_{G/K}:  |{{\mathcal{C}}}_{G/K}| \times K(R_{S^1}(\widetilde{K}),n) \to K(R_{S^1}(\widetilde{K}),n)$$
    which are ${\rm{Hom}}(K,S^1)$ equivariant, and therefore they define maps
    $$({\mathfrak{TR}}_n)_{G/K} \to ({\mathfrak{TR}}_{n+3})_{G/K}$$
    over $|{{\mathcal{C}}}_{G/K}|$ which can be assembled into a map ${\mathfrak{TR}}_n \to {\mathfrak{TR}}_{n+3}$ over $|{\mathcal{C}}|$.
At the level of based maps we have an induced map
\begin{eqnarray*}
{\rm{Hom}}_{\mathcal{O}_G^P} (\Phi X , {\mathfrak{TR}}_n)_{|\mathcal{C}|}  &\to&  {\rm{Hom}}_{\mathcal{O}_G^P} (\Phi X , {\mathfrak{TR}}_{n+3}) _{|\mathcal{C}|} \\
F  &\mapsto& \widetilde{F}
\end{eqnarray*}
   with $\widetilde{F}_{G/K}(x): =\varepsilon_{G/K}(\mu_{G/K} (\kappa_{G/K}(x)), F(x))$, such that it  induces a degree three homomorphism
   $$\eta \cup: \mathbb{H}^n_G(\Phi X, \mu) \to \mathbb{H}^{n+3}_G(\Phi X,\mu).$$
   
   \begin{remark} The procedure described above defines in general a $H^*(X,\mathbb{Z})^G$-module
   structure  on $\mathbb{H}^*_G(\Phi X, \mu)$ by the cup product. Therefore we could say that the degree three homomorphism
   $\eta \cup$ is equivalent to performing the cup product with the class $\eta$.
   
   If the group $G$ is trivial, the class $\eta \in H^3(X,\mathbb{Z})$ classifies the projective unitary bundles over $X$, and it was proven
   by Atiyah and Segal \cite{Atiyah-Segal2} that the third differential of Segal's spectral sequence was equivalent to the homomorphism $Sq^3_{\mathbb{Z}} - \eta \cup$.
  \end{remark}

  \begin{theorem}\label{theoremspectralsequence}
  Consider the Segal's spectral sequence defined in Theorem \ref{theorem second page is Bredon} and the isomorphism of its second page with the Bredon cohomology with coefficients in twisted representations
  $$E_2^{p,q} \cong \mathbb{H}^p_G(\Phi X,\mu)$$
  whenever $q$ is even. Then the third differential of the spectral sequence $d_3 : E_2^{p,q} \to E_2^{p+3,q-2}$ is  a  natural  transformation   in Bredon  cohomology  with  local coefficients in  twisted  representations.

  \end{theorem}
  
\begin{proof}

The result follows from Brown's representability theorem  (see section \ref{subsection Brown representability}  in the Appendix  for  a discussion of  Brown  representability in the parametrized  and  equivariant  setting). 
Since the  third differential is a homomorphism
$$ \mathbb{H}^p_G(\Phi X, \mu )  \to  \mathbb{H}^{p+3}_G(\Phi X,\mu)$$
which is functorial and only depends on the map $\mu:\Phi X \to |\mathcal{C}|$,
the  third  differential  is  thus given  by  
 a map of ${\mathfrak{TR}}_p\to {\mathfrak{TR}}_{p+3}$ of $\OO_G^ P$-spaces over $|{\mathcal{C}} |$. 
 \end{proof}
 Note that a  map
 from $({\mathfrak{TR}}_p)_{G/K}\to ({\mathfrak{TR}}_{p+3})_{G/K}$ over $|{{\mathcal{C}}}_{G/K}|$ is determined
 by a ${\rm{Hom}}(K,S^1)$-equivariant map
 $$\overline{| {{\mathcal{C}}}_{G/K}|_{\widetilde{K}}}\times K(R_{S^1}(\widetilde{K}),n) \to K(R_{S^1}(\widetilde{K}),n+3).$$
 The assembly of these maps produces a map ${\mathfrak{TR}}_p\to {\mathfrak{TR}}_{p+3}$ of $\OO_G^ P$-spaces over $|{\mathcal{C}} |$.

 In the case that the  acting  group is  trivial, Atiyah and Segal have proved \cite[Prop. 4.6]{Atiyah-Segal2} that the map of Eilenberg-MacLane spaces
$$K(\mathbb{Z},3) \times K(\mathbb{Z},p) \to K(\mathbb{Z},p+3)$$
is given by the operation $(\eta, b) \mapsto Sq^3_{\mathbb{Z}}b - \eta\cup b$. 

Equivariantly, the  situation  is  much  more  involved. A  complete  description  of  natural  transformations   in  Bredon  cohomology  with  local  coefficients  is  not  available  in the  literature.

Even  untwisted, the expression  for  the  third  differential  in Bredon  cohomology  turns  out  to  be considerably  different.    One  could  expect  that a version  of  Steenrod  cubes  defined  as  follows  should  cover  the  natural  transformations.

 For any $S^1$ central extension $\widetilde{K}$, the group of twisted representations $R_{S^1}(\widetilde{K})$ is a free $\mathbb{Z}$-module generated by the irreducible representations of $\widetilde{K}$ on which  $S^ {1}$  acts  by  scalar  multiplication. Therefore we have the short exact sequence of coefficients
  \begin{equation} \label{short exact sequence of coefficients} 0 \to R_{S^1}(\widetilde{K}) \stackrel{\times 2}{\longrightarrow} R_{S^1}(\widetilde{K})\stackrel{{\rm{mod}}_2}{\longrightarrow} R_{S^1}(\widetilde{K}) \otimes_{\mathbb{Z}} \mathbb{Z}/2 \to 0 \end{equation}
  and we can consider the composition of maps of $\Omega$-spectra

  $$HR_{S^1}(\widetilde{K}) \stackrel{{\rm{mod}}_2}{\longrightarrow} H (R_{S^1}(\widetilde{K})\otimes_{\mathbb{Z}} \mathbb{Z}/2 )
 \stackrel{{\rm{Sq}}^2}{\longrightarrow} \Sigma^2 H(R_{S^1}(\widetilde{K})\otimes_{\mathbb{Z}} \mathbb{Z}/2) \stackrel{\beta}{\longrightarrow} \Sigma^3 HR_{S^1}(\widetilde{K})$$
  where the first map is the reduction modulo 2 map, the second is the Steenrod square defined over each $\mathbb{Z}/2$-module
  generated by irreducible representations, and the third map is the Bockstein map induced by the short exact sequence of \eqref{short exact sequence of coefficients}.
  
  Denote the composition 
  $${\rm{Sq}}^3_{\widetilde{K}}= \beta \circ {\rm{Sq}}^2 \circ {\rm{mod}}_2 : HR_{S^1}(\widetilde{K}) \to \Sigma^3 HR_{S^1}(\widetilde{K})$$
  and note that it is compatible with the $N_G(K)/K$-action on $HR_{S^1}(\widetilde{K})$ and with the restriction maps.
 At the level of the $\mathcal{O}_G^P$-spaces over $|{\mathcal{C}}|$
we see that the maps ${\rm{Sq^3_{\widetilde{K}}}}$ induce maps 
   $$\xymatrix{ \bigsqcup_{\widetilde{K} \in {\rm{Ext}}(K,S^1)} \overline{|{{\mathcal{C}}}_{G/K}|_{\widetilde{K}}} \times_{{\rm{Hom}}(K,S^1)} K(R_{S^1}(\widetilde{K}),n) \ar[d]^{{\rm{Sq^3_{\widetilde{K}}}}}\\
  \bigsqcup_{\widetilde{K} \in {\rm{Ext}}(K,S^1)} \overline{|{{\mathcal{C}}}_{G/K}|_{\widetilde{K}}} \times_{{\rm{Hom}}(K,S^1)}K(R_{S^1}(\widetilde{K}),n+3)}$$
  which can be assembled into a map that we denote
  $$\mathrm{Sq}^3: {\mathfrak{TR}_n} \to {\mathfrak{TR}_{n+3}},$$
  which furthermore assembles into a map of $\mathcal{O}_G^P$-spectra over $|{{\mathcal{C}}}|$ which we denote
  $${\rm{Sq^3}}: {\mathfrak{TR}} \to \Sigma^3 {\mathfrak{TR}}.$$

At the level of  Bredon cohomology with local coefficients in twisted representations, the map ${\rm{Sq^3}}$ induces a degree
three homomorphism
\begin{align}\label{def-Sq3}
{\rm{Sq^3}}: \mathbb{H}^p_G(\Phi X, \mu) \to \mathbb{H}^{p+3}_G(\Phi X,\mu)
\end{align}
which will be denoted by the same symbol in order to simplify the notation.

The Steenrod  cube  over  twisted  representation vanishes  on zero dimensional  Bredon  cohomology  classes. The coincidence  of  the  third  differential   for  the  spectral  sequence  with  this  cohomology  operation  would  imply  that the  edge  homomorphism  

$$K^0_G(X) \to E_\infty^{0,2r} \to E_2^{0,2r}\cong H^0_G(X;R(?))$$
is  surjective. 
However, evidence  in  specific  computations \cite[Ex.  5.2 pp. 614]{lueckolivervectorbundles} and \cite[Lemma  3.3  pp. 6]{degrijselearyvectorbundles} shows  that  this  is  not  the  case. The  first author thanks Dieter  Degrijse  and Justin  Noel  for  conversations  on  this  issue leading  to  a  precision  on  the first  version  of  this  note.

\section{Applications}\label{section_applications}

\subsection{Equivariant K-theory}
 When Segal's spectral sequence is applied to non-twisted equivariant K-theory, it is known that the second page of the spectral sequence is isomorphic to the Bredon cohomology with coefficients in representations
$$E_2^{p,q} = \mathbb{H}_G^p(X, \mathcal{R}( ?))$$
where $\mathcal{R}(G/K)=R(K)$ is the representation  ring  of $K$. 

\subsection{The case of $\eta=0$}
If the restriction of the class $\overline{\eta} \in H^3(X \times_G EG; \mathbb{Z})$ to $H^3(X; \mathbb{Z})$ is zero, then we have that all the higher differentials of Segal's spectral sequence vanish if we tensor the spectral sequence with the rationals. This follows from the fact that the operations on the Eilenberg-MacLane spectrum are all torsion operations. In this case Segal's spectral sequence tensored with the rationals collapses at the second page, and therefore the twisted equivariant K-theory is isomorphic to the Bredon cohomology with local coefficients in twisted representations after tensoring both cohomology groups with the rationals.

\subsection{Twisted equivariant  K-theory  for  trivial $G$-spheres}
We know from \cite[Theorem 4.8]{barcenasespinozajoachimuribe} that the twistings are classified by $H^3(X\times_GEG ;\IZ)$. In the case that $X$ is a trivial $G$-space, we have that the  group $G$  is  finite and the Borel cohomology  group  satisfies
$$H^3(X\times_GEG;\IZ)\cong H^3(X\times BG;\IZ),$$
and if $X$ has torsion free integral cohomology, by the K\"unneth isomorphism we obtain
$$H^3(X\times BG;\IZ)\cong \bigoplus_{i=0}^3 H^i(X;\IZ)\otimes H^{3-i}(BG;\IZ).$$

In the case that  $X=S^1$, given $$\alpha = [P] \in H^3(S^1 \times_G EG;\IZ)\cong H^2(BG;\IZ)\oplus H^3(BG;\IZ),$$ the class $\alpha$ can be decomposed as $\alpha=\gamma \oplus \beta$, with
$\gamma\in H^2(BG;\IZ)\cong {\rm{Hom}}(G,S^1)$ and $\beta\in H^3(BG;\IZ)\cong \mathrm{Ext}(G,S^1)$. To the homomorphism 
$\gamma : G \to S^1$ one can associate the linear 1-dimensional representation $\rho_\gamma$, and let 
$1\to S^1\to \widetilde{G}\to G \to 1 $ be  the $S^1$-central  extension of $G$ associated  to  $\beta$.

If $U$ and $V$ are two open contractible subsets of $S^1$, with $U\cup V = S^1$ and $U\cap V \simeq S^0$, then the Mayer-Vietoris sequence for $ K_G^*(S^1;P)$ is given by the following six-terms exact sequence
$$\xymatrix{K_G^0(S^1;P) \ar[r] & K_G^0(U;P|_U)\oplus K_G^0(V;P|_V) \ar[r] & K_G^0(U \cap V;P|_{U\cap V}) \ar[d] \\ K_G^1(U \cap V;P|_{U\cap V}) \ar[u] & K_G^1(U;P|_U)\oplus K_G^1(V;P|_V) \ar[l] & K_G^1(S^1;P). \ar[l] }$$
On the other hand (cf. \cite[Section 5.3.4]{barcenasespinozajoachimuribe}), the isomorphisms
 $K_G^0(U;P|_U)\cong R_{S^1}(\widetilde{G}) \cong K_G^0(V;P|_V)$ and $K_G^0(U \cap V;P|_{U\cap V}) \cong
  R_{S^1}(\widetilde{G}) \oplus R_{S^1}(\widetilde{G})$ fit in the following commutative diagram 
$$\xymatrix{ K_G^0(U;P|_U)\oplus K_G^0(V;P|_V) \ar[d]^\cong \ar[r] & K_G^0(U \cap V;P|_{U\cap V}) \ar[d]^\cong \\
 R_{S^1}(\widetilde{G}) \oplus  R_{S^1}(\widetilde{G}) \ar[r]^-{j^*} &  R_{S^1}(\widetilde{G}) \oplus R_{S^1}(\widetilde{G})},$$
where the bottom morphism $j^*: (a,b) \mapsto (a-b,a-\rho_\gamma \cdot b)$ is induced by the inclusions $U\cap V \hookrightarrow U$ and $U\cap V \hookrightarrow V$;  thus we obtain the exact sequence
$$\xymatrix{
0 \ar[r] & K_G^0(S^1;P) \ar[r] &  R_{S^1}(\widetilde{G}) \ar[r]^{\times(1-\rho_\gamma)} & R_{S^1}(\widetilde{G}) \ar[r] & K_G^1(S^1;P) \ar[r] & 0}$$
which implies that the K-theory groups are respectively the invariants
and the coinvariants of the operator $\rho_\gamma$, i.e.
$$K_G^0 (S^1;P) \cong 
R_{S^1}(\widetilde{G})^{\rho_\gamma}  \hspace{0.5cm} \text{and}  \hspace{0.5cm} K_G^1 (S^1;P) \cong R_{S^1}(\widetilde{G}) / (1-\rho_\gamma) R_{S^1}(\widetilde{G}).$$

For the 2-dimensional sphere, the Borel cohomology is given by $$H^3(S^2\times_G EG ; \mathbb{Z}) \cong H^3(BG;\mathbb{Z})$$ by K\"unneth formula and the fact $H^2(S^2;\mathbb{Z}) \otimes H^1(BG;\mathbb{Z})=0$, since $G$ is finite. So, in this case there is only discrete torsion and $$K_G^*(S^2;P) \cong K^*(S^2) \otimes R_{S^1} (\widetilde{G}).$$
where $\widetilde{G}$ is the $S^1$-central extension associated to $[P]$.

For $X=S^3$ with a trivial $G$-action, we have 
$$H^{3} (S^{3} \times_G EG ;\mathbb Z) \cong H^{3}(S^{3};\mathbb Z) \oplus H^{3}(BG;\mathbb Z),$$
thus every cohomology class $\alpha \in H^{3}_G (S^{3} ;\mathbb Z)$ can be decomposed as  $n\gamma \oplus \beta$, where $\gamma \in H^{3}(S^{3};\mathbb Z)$ is the generator
and $\beta \in H^{3}(BG;\mathbb Z)$. Take a projective unitary stable bundle $P$ over $S^3$ which is classified by the class 
$n\gamma \oplus \beta$. Then in this case the second page of Segal's spectral sequence is isomorphic to 
$$H^*(S^3) \otimes_{\mathbb{Z}} R_{S^1}(\widetilde{G})$$
and the third differential is given by cupping with the class $n \gamma \otimes 1$. Therefore we get that for $n \neq 0$
$$K_G^0 (S^3;P)=0   \hspace{0.5cm} \text{and}  \hspace{0.5cm} K_G^1 (S^3;P) \cong \mathbb{Z}/n \otimes_{\mathbb{Z}} R_{S^1}(\widetilde{G}).$$

\subsection{Discrete torsion} One of the first versions of twisted equivariant K-theory were defined with the information
of a 2-cocycle $Z^2(G, S^1)$ whenever the group was finite (see \cite{Witten,LupercioUribe} and references therein); these cocycles were called {\it discrete torsion}. Using the fact that the group $H^2(G, S^1)$ classifies isomorphism classes of $S^1$-central extensions of the group $G$, this
definition of the twisted equivariant K-theory was generalized to the context of proper and discrete actions in \cite{Dwyer},  under   the  additional  hypothesis that  the  class  $\eta \in H^2(G, S^1)$ is  a  finite  order  element.
With our setup we can recover the twisted equivariant K-theory groups associated to discrete torsion, as  well as  the spectral sequence developed in \cite{Dwyer}. 

Let $G$ be a countable discrete group and let $1 \to S^1 \to \widetilde{G} \to G \to 1$ be a $S^1$-central extension of $G$ which is classified by the cohomology class $\alpha \in H^2(G,S^1)$.
Consider $L^2(\widetilde{G})$, the square integrable complex functions on $\widetilde{G}$, and endow it with the natural $\widetilde{G}$-action given by composition
$(g \cdot f)(h):= f(hg^{-1}).$
Let $$V(\widetilde{G}) := \{ f \in L^2(\widetilde{G}) \mid f(hx) =f(h)x \text{ for all }h \in \widetilde{G} \text{ and } x \in S^1\}$$
be the subspace on which $S^1$ acts by multiplication and let $\mathcal{H}: =V(\widetilde{G}) \otimes L^2([0,1])$
be the $\widetilde{G}$-Hilbert space on which kernel $\widetilde{G} \to G$ acts also by multiplication. Let
 $\mathcal{U}( \mathcal{H})$ be the group of unitary operators on $\mathcal{H}$ and note that
the $\widetilde{G}$-action on $V(\widetilde{G})$ defines a homomorphism
$$\widetilde{\rho}: \widetilde{G} \to \mathcal{U}( \mathcal{H})$$
whose projectivisation $\rho : G \to P\mathcal{U}( \mathcal{H})$ makes the following diagram commutative
$$\xymatrix{\widetilde{G} \ar[r]^-{\widetilde{\rho}} \ar[d] & \mathcal{U}(\mathcal{H}) \ar[d] \\
G \ar[r]^-{\rho} & P\mathcal{U}( \mathcal{H}).
}$$

For every orbit type $G/K$ with $K$ finite, define the functor
$$\rho_{G/K}:=\Funct_{st} (G \ltimes G/K, P \mathcal{U}( \mathcal{H}))$$
by the equation $\rho_{G/K}(g,h[K]) := \rho(g)$, and note that this assignment is functorial since any $G$-equivariant map $\psi: G/K \to G/H$ induces a functor $G \rtimes G/K \to G \rtimes G/H, (g,h[K]) \mapsto (g, \psi(h[K]))$, and therefore the first coordinate stays fixed.
Moreover, since $L^2(\widetilde{K}) \subset L^2(\widetilde{G})$, where $\widetilde{K}$ denotes the $S^1$-central extension of $K$ induced by $\widetilde{G}$ and the inclusion $K \subset G$, then we know by Peter-Weyl's theorem that $\calh$ includes all irreducible representations of $\widetilde{K}$ on which the circle acts by multiplication, an infinitely number of times; therefore we know that $\rho_{G/K}$ is a stable functor, since its restriction to the group $K$ is a stable homomorphism (see Definition \ref{def:proj_equi-bundle}),  and therefore it defines a point in $|\mathcal{C}_{G/K}|$.

For every proper $G$-CW-complex $X$ we can associate the map of $\mathcal{O}_G^P$-spaces
$$\overline{\rho}^X : \Phi X \to |\mathcal{C}|$$
such that for every orbit type we get the constant map
$$\overline{\rho}_{G/K}^X : X^K \mapsto |\mathcal{C}_{G/K}|, \hspace{0.5cm} x \mapsto \rho_{G/K}.$$

In this way we get that 
the twisted $G$-equivariant K-theory groups $K_G^{*}(\Phi X; \overline{\rho}^X)$ realize the twisted $G$-equivariant K-theory groups
associated to the $S^1$-central extension $\widetilde{G}$ defined by Dwyer  in \cite{Dwyer}. Now, since the map $ \overline{\rho}^X$
is constant on each orbit type and only depends on the central extension $\widetilde{G}$ defined by $\alpha$, we could define the  
contravariant $\mathcal{O}_G^P$-module 
$\mathcal{R}_{\alpha}(?) $ with $\mathcal{R}_{\alpha}(G/K):=R_{S^1}(\widetilde{K}) $ thus obtaining a canonical isomorphism
$$\mathbb{H}^*_G(\Phi X,\overline{\rho}^X )\cong  \check{H}^{*}_G(X;\mathcal{R}_{\alpha}(?) )
$$
between the Bredon cohomology of the map $ \overline{\rho}^X$ and the Bredon cohomology with coefficients in the twisted representations $\mathcal{R}_{\alpha}(?)$.

The groups $\check{H}^{*}_G(X;\mathcal{R}_{\alpha}(?) )$ are the ones shown in \cite{Dwyer} to be isomorphic to the second page 
of the Atiyah-Hirzebruch spectral sequence that converges to the twisted equivariant K-theory groups $K_G^{*}(\Phi X; \overline{\rho}^X)$.

The methods developed in the present work have been successfully applied in \cite{barcenasvelasquez} for the
explicit calculation of the  twisted  $Sl_{3}(\mathbb{Z})$-equivariant  $K$-theory   and  $K$-homology of the space $\eub{Sl_{3}(\mathbb{Z})}$. In this case, the calculations are done using an universal  coefficients  theorem  for  $\alpha$-twisted  Bredon  cohomology, and  the fact that   the  spectral  sequence  constructed  in this work collapses at the second page.

\section{Appendix: Brown representability} \label{Appendix:brown_rep}

The content of this appendix is based on Chapter 7 of \cite{maysigurdsson}. We assume the reader is familiar  with the  $qf$-model category structure defined in \cite[Section 6.2]{maysigurdsson}. 

\subsection{Based $G$-CW-complexes} Let $ B$ denote a fixed proper $G$-CW-complex. A based proper $G$-CW-complex is a pair $(X;x)$ with $X$ a $G$-CW-complex, $X-\{x\}$ a proper $G$-CW-complex and $x$ a $G$-fixed point.

A based $G$-space over $ B$ is a triple $X=(X,p,s)$ where $p:X\to  B$ and $s: B \to X$ are $G$-maps and $p\circ s = \id_B$. A map $X\to X'$ of based $G$-spaces over $ B$ is a map of based $G$-spaces that commute with projections and sections. We denote the space of such maps by $\Hom_{G,B}^0(X,X')$ and by $_G[X,X']_{B}^0$ the corresponding set of homotopy classes.

Let $(X,p)$ be a $G$-space over $ B$. We use the notation $(X,p)_+$ for the union $X\coprod  B$ of a based $G$-space $(X,p)$ over $ B$ with a disjoint section, i.e. $(X,p)_+=(X\coprod  B, p\coprod \id, i)$, where $i: B\to X\coprod  B$ is the natural inclusion.

If $(X,p)$ is a $G$-space over $B$ and $Z$ is a based $G$-space, then let $X \times_B Z$ be the $G$-space $X \times Z$ with projection the product of the projections $p:X \to B$ and $Z\to *$. Define $X\wedge_B Z$ to be the quotient of $X \times_B Z$ obtained by taking fiberwise smash products, so that $(X \wedge_B Z)_b = X_b \wedge Z$; the basepoints of fibers prescribe the section.

For $G$-spaces $(X,p)$ and $(Y,q)$ over $B$, $X \times_B Y$ is the pullback of the projections $p:X \to B$ and $q:Y\to B$, with the evident $G$-projection $X \times_B Y \to B$. When $X$ and $Y$ have $G$-equivariant sections $s$ and $t$, their pushout $X \vee_B Y$ specifies the coproduct, or wedge, of $X$ and $Y$ in the category of based proper $G$-spaces, and $s$ and $t$ induce a $G$-map $X \vee_B Y \to X \times_B Y$ over $B$ that sends $x$ and $y$ to $(x, tp(x))$ and $(sq(y), y)$. Then $X \wedge_B Y$ is the pushout in the category of compactly generated spaces over $B$, displayed in the diagram $$\xymatrix{X \vee_B Y \ar[r] \ar[d] & X \times_B Y \ar[d] \\ {*}_B \ar[r] & X \wedge_B Y.}$$ This implies that $(X \wedge_B Y )_b = X_b \wedge Y_b$, and the section and projection are evident maps.

We denote by $\Sigma_B X$ the $G$-space $S^1\wedge_B X$ over $B$, where $S^1$ has the trivial $G$-action. 

\subsection{$qf$-model category structure for $G$-CW-complexes over $B$}
Let $n$ be a natural number. Let $I^G$ be the set of all maps of the form $G/H_+ \times i$, where $H$ is a finite subgroup of $G$ and $i$ runs through the set of based inclusions $i:S^{n-1}_+\to D^n_+$ (where $S^{-1}$ is empty). Analogously, let $J^G$ be the set of all maps of the form $G/H_+ \times i_0$, where $H$ is a finite subgroup of $G$ and $i_0$ runs through the set of based maps $i_0:D^n_+\to (D^n\times I)_+$.

Given maps $i:(X,p)\to (Y,q)$ and $d:(Y,q)\to B$ of based $G$-CW-complexes, the composition $d\circ i:(X,p)\to B$ defines $i$ as a map over $B$. We write $i(d)$ for this map over $B$.
Let $I^G_B$ be the set of all such maps $i(d)$ with $i\in I^G$, and denote by $J_B^G$ the set of all such maps $j(d)$ with $j\in J^G$.

In order to define the $qf$-model category structure on proper $G$-CW-complexes over $B$ we need to recall the definition of $q$-fibration.

\begin{proposition}{\cite[Prop. 6.2.2]{maysigurdsson}}
The following conditions on a map of compactly generated spaces $p:E\to Y$ are equivalent. If they are satisfied, then $p$ is called a q-fibration.
\begin{enumerate}
\item The map $p$ satisfies the covering homotopy property with respect to disks $D^n$; that means there is a lift in the following diagram
  $$\xymatrix{
  D^n\ar[r]^{\alpha}\ar[d]&E\ar[d]^{p}\\
  D^n\times I\ar@{-->}[ur]\ar[r]^{ \ \ \  h}&Y.}$$
\item If $h$ is a homotopy relative to the boundary $S^{n-1}$ in the diagram above, then there is a lift that is a homotopy relative to the boundary.
\item The map $p$ has the relative lifting property (RLP) with respect to the inclusion $S^n_+\to D^{n+1}$
of the upper hemisphere into the boundary $S^n$ of $D^{n+1}$; that is, there is a lift in the
diagram
  $$\xymatrix{
  S^n_+\ar[r]^{\alpha}\ar[d]&E\ar[d]^p\\
  D^{n+1}\ar@{-->}[ur]\ar[r]^{ \ \ \ \bar{h}}&Y.}$$
\end{enumerate}
\end{proposition}

\begin{definition}
A map $g$ of spaces over $B$ is an $f$-cofibration if it satisfies the fiberwise homotopy extension property (HEP), that is, if it has the left lifting property (LLP) with respect to the maps $p_0: \mathrm{Map}_B (I,X) \to X$.
 \end{definition}

A map $d:D^n\to B$ of compactly generated spaces is said to be an $f$-disk if $i(d):S^{n-1}\to D^n$ is an $f$-cofibration. An $f$-disk $d:D^{n+1}\to B$ is said to be a relative $f$-disk if the lower hemisphere $S^n_-$ is also an $f$-disk, so that the restriction $i(d):S^{n-1} \to S^n_-$ is an $f$-cofibration.

A map $f: (X,p,s)\to (Y, q,t)$ of based $G$-spaces over $B$ is called a $q$-equivalence if $f :X \to Y$ is a $G$-equivariant weak equivalence of spaces (forgetting the based structure over $B$ ).
 Define $I_B^f$ to be the set of inclusions $i(d):S^{n-1}\to D^n$ where $d:D^n\to B$ is an $f$-disk. Define $J_B^f$ to be the set of inclusions $i(d):S^n_+\to D^{n+1}$ of the upper hemisphere into a relative $f$-disk $d:D^{n+1}\to B$. A map of compactly generated spaces over $B$ is said to be 
 \begin{enumerate}
 \item a $qf$-fibration if it has the RLP with respect to $J_B^f$ and
 \item a $qf$-cofibration if it has the LLP with respect to all $q$-acyclic $qf$-fibrations, that is, with respect to those $qf$-fibrations that are $q$-equivalences.
 \end{enumerate}

Now we proceed equivariantly. Let $\OO_G^\calall$ denote the set of all orbits $G/H$.
\begin{definition}
A set $\calc$ of proper $G$-CW-complexes that contains the orbits $G/K$ with $K\in\calfin(G)$ and is closed under products with elements in $\OO_G^\calall$ is called a generating set. It is closed if it is closed under finite products.
\end{definition}

Let $\calc$ be a generating set.
\begin{enumerate}
\item Let $I_B^f(\calc)$ be the set of maps $$(\id\times i)(d)\coprod \id:C\times S^{n-1}\coprod B\to C\times D^n\coprod B$$such that $C\in\calc$, $d:C\times D^n\to B$ is a $G$-map, $i$ is the boundary inclusion, and the associated map $\widetilde{i}$ over $\mathrm{Map}_G(C, B)$ is a generating  $qf$-cofibration in the category of compactly generated spaces over $\mathrm{Map}_G(C, B)$.
\item Let $J_B^f(\calc)$ consist of the maps $$(\id\times i)(d)\coprod \id:C\times S^{n}_+\coprod  B\to C\times D^{n+1}\coprod  B$$such that $C\in\calc$, $d:C\times D^{n+1}\to  B$ is a $G$-map, $i$ is the inclusion, and the associated map $\widetilde{i}$ over $\mathrm{Map}_G(C, B)$ is a generating acyclic $qf$-cofibration in the category of compactly generated spaces over  $\mathrm{Map}_G(C, B)$.
\end{enumerate}

Fixing a generating set $\calc$, we define a $qf$-type model structure based on $\calc$, called the $qf(\calc)$-model structure. Its weak equivalences are the weak equivalences of proper $G$-CW-complexes. We define now the $qf(\calc)$-fibrations.
\begin{definition}
A $G$-map over $ B$ is a $qf(\calc)$-fibration if $\mathrm{Map}_G(C,f)$ is a $qf$-fibration in the category of compactly generated spaces over $\mathrm{Map}_G(C,B)$, for all $C\in\calc$.
\end{definition}
In \cite[Section 5.5, p.90]{maysigurdsson}, the notion of \emph{well grounded} model category is  introduced. There it  is  established  that the category of based proper $G$-CW-complexes can be endowed with a structure of a well grounded model category. 
\begin{theorem}{\cite[Th. 7.2.8]{maysigurdsson}}\label{wellgrounded}
For any generating set $\calc$ the category of based proper $G$-CW-complexes over $B$ is a well grounded model category. The weak equivalences  are the based weak $G$-homotopy equivalences and fibrations are $qf(\calc)$-fibrations. The sets $I_B^f(\calc)$ and $J_B^f(\calc)$ are the generating $qf(\calc)$-cofibrations and the generating acyclic $qf(\calc)$-cofibrations. 
\end{theorem}

We define a $qf$-fibration in the category of based $G$-spaces over $B$ as a map which is a $qf$-fibration when regarded as a map of $G$-spaces over $B$, and similarly for $qf$-cofibrations.

Let $(Y,q,t)$ is a based $G$-space over $B$ and $f:A\to B$ be a $G$-map, we define $f^*Y$ as the based $G$-space over $A$ obtained from the pullback diagram
$$\xymatrix{
	A\ar[d]^{s}\ar[r]^{f}&B\ar[d]^{t}\\
	f^*Y\ar[d]^{p}\ar[r]&Y\ar[d]^{q}\\
	A\ar[r]^{f}&B
	}$$
	
On the other hand if $(X,s,p)$ is a based $G$-space over $A$ and $f:A\to B$, define $f_*X$ and its structure maps $q$ and $t$ by means of the
map of retracts in the following diagram on the left, where the top square is a
pushout and the bottom square is defined by the universal property of pushouts
and the requirement that $q\circ t=id$.

$$\xymatrix{
	A\ar[d]^{s}\ar[r]^f&B\ar[d]^{t}\\
	X\ar[d]^{p}\ar[r]&f_!X\ar[d]^{q}\\
	A\ar[r]^f&B
}$$\vspace{0.5cm}

Recall that an adjoint pair of functors $(T,U)$ between model
categories is a \emph{Quillen adjoint pair}, or a \emph{Quillen adjunction}, if the left adjoint $T$ preserves cofibrations and acyclic cofibrations or, equivalently, the right adjoint $U$ preserves fibrations and acyclic fibrations. It is a \emph{Quillen equivalence} if the induced adjunction on homotopy categories is an adjoint equivalence.

The  following  base  change  result  will  be  useful  in  the  following: 

\begin{proposition}\label{prop: quilllen equiv rectrac}. If $f: Z_1 \rightarrow Z_2$ is a $G$-map, then the pair $(f_{\ast},f^{\ast})$ is a Quillen adjunction for the model category structures defined above. Morever, if $f$ is a weak  $\mathcal{F}$-equivalence with  respect  to  a family  of  subgroups  $\mathcal{F}$,  then $(f_{\ast},f^{\ast})$ is a Quillen equivalence

\end{proposition}
\begin{proof} This is the content of Propositions 7.3.4 and 7.3.5 in \cite{maysigurdsson}. Note that $f_{\ast}$ is denoted by $f_{!}$ in \cite{maysigurdsson}.

\end{proof}

\begin{theorem}
Let  $G$  be a  discrete  group. For  a $G$-CW  complex  $B$,    there  exists  a  ziz-zag  of  Quillen equivalences  between  the  category  of $G$-spaces  over  $B$ with  the  $qf$-model  structure  and   $\OO_G^P$-spaces  over  $\Phi B$  with the  levelwise  model structure. 

\end{theorem}
\begin{proof}

Let  $\Phi$  be  the  fixed  point  functor, which  associates  an $\OO_G^P$-complex to a  $G$-CW complex $X$. Notice  that   the  additional  section  or  projection  data $p: X\to B$,  and $s:B\to X$  restrict to  fixed  points.

Let $J$  be  a functorial cellular  approximation functor in the  category  of $\OO_G^P$-spaces (See  Theorem  3.7  in  \cite{lueckdavis}),  in the  sense  that  $J(\XX)$  is  a free  $\OO_G^P$-CW complex  for  every  $\OO_G^P$-space $\XX$,  and $J(\XX)\to \XX$ is  a  weak  $\OO_G^P$-equivalence .
 
The  cellular  approximation  functor  defines  a map $\epsilon: J( \Phi B) \times_{\OO_G^P} \nabla \to B$,  which is  a  $G$-homotopy  equivalence.  The  base  change  functors $(\epsilon_*, \epsilon^*)$  satisfy the  hypothesis  of Proposition \ref{prop: quilllen equiv rectrac} thus  giving  a  Quillen equivalence pair   between  the  categories  of  $G$-spaces  over $B$ and over $J( \Phi B) \times_{\OO_G^P} \nabla $.

Analogously, the mentioned cellular  approximation defines  a map  of  $\OO_G^P$-spaces $\epsilon^{'}: J\Phi(B)   \to \Phi (B)$   giving  a  Quillen  equivalence  ${\epsilon^{'}}_*, {\epsilon^{'}}^{*}$ between  the  categories  of  $\OO_G^P$-spaces  over $\Phi B$ and over $J( \Phi B)$. 

Finally,  the $\OO_G^P$ map $\Phi B\to J(\Phi B)$ determines  a Quillen equivalence  pair  between  the  categories of  $\OO_G^P$ spaces over $\Phi B$  and $G$-spaces  over  $J(\Phi B)\times_{\OO_G^P} \nabla$.

\end{proof}

\subsection{Generating sets in the category of proper $G$-CW-complex  over $ B$.}

Let $\mathfrak{H}$ be a category with weak colimits, denoted by $\hocolim Y_n$, we say that an object $X$ of $\mathfrak{H}$ is compact if
$$\colim \mathfrak{H}(X,Y_n)\cong \mathfrak{H}(X,\hocolim(Y_n))$$
for any sequence of maps $Y_n\to Y_{n+1}$ in $\mathfrak{H}$.
\begin{definition}
	A set $\cald$ of objects in a pointed category $\mathfrak{H}$ is a generating set if a map $f:X\to Y$ such that $f_*:\mathfrak{H}(D,X)\to \mathfrak{H}(D,Y)$ is a bijection for all $D\in\cald$ is an isomorphism.
\end{definition}

For $n>0$, $b\in B$, and $H\subset G_b$, let $S^{n,b}_H$ be the based $G$-space over $B$ given by $(G/H_+\wedge S^n)\vee_bB$, where the wedge is taken with respect to the standard basepoint of $G/H_+\wedge S^n$ and the basepoint $b\in B$. The inclusion of $B$ gives the section and the projection maps $G/H_+\wedge S^n$ to the point $b$ and maps $B$ by the identity map.

Let $\mathbb{D}_B^c$ be the set of all such based $G$-spaces $S_H^{n,b}$ over $B$, with $n>0$. Then, from \cite[Lemma 7.5.13-14]{maysigurdsson} it follows the next result.

\begin{lemma}
$\mathbb{D}_B^c$ is a generating set in the homotopy category of based $G$-connected spaces over $B$. Moreover, each element in $\mathbb{D}_B^c$ is a compact object in that category.
\end{lemma}
We want to use the following abstract Brown representability theorem.
\begin{theorem}{\cite[Th. 7.5.7]{maysigurdsson}}\label{abstractbrown}
Let $\mathfrak{H}$ be a category with coproducts and weak pushouts. Assume that $\mathfrak{H}$ has a generating set of compact objects. Let $k:\mathfrak{H}\to \SETS$ be a contravariant functor that takes coproducts to products and weak pushouts to weak pullbacks. Then there is an object $Y\in\mathfrak{H}$ and a natural isomorphism $k(X)\cong \mathfrak{H}(X,Y)$ for $X\in \mathfrak{H}$.
\end{theorem}

Given a model category $\mathfrak{T}$, it is possible to construct the \emph{homotopy category} $\mathfrak{H}$. For its definition see \cite{DwyerSpalinski}.
\subsection[Equivariant parametrized homotopy theories and Brown representability]{Equivariant parametrized homotopy theories and Brown representa-bility} \label{subsection Brown representability}

Let us recall that a functor $\calh_G^B$ defined on the category of based proper $G$-CW-complexes over $ B$ with values in $\IZ$-modules is a proper reduced generalized cohomology theory over $ B$ if it satisfies a parametrized version of the Eilenberg-Steenrod axioms (except the reduced dimension axiom). For  definitions of axioms see for example \cite[Section 4.3.2]{barcenasespinozajoachimuribe} or \cite[Definition 20.1.2]{maysigurdsson}.

\begin{theorem}
Let $\calh_G^*$ be a reduced proper $G$-equivariant parametrized cohomology theory over $ B$. Then, there exist a sequence of proper $G$-CW-complexes $\mathcal{R}_n$ over $ B$ and natural transformations such that
$$\calh_G^n(X,p,s)\cong\,  _G[X,\mathcal{R}_n]_{B}^0$$
for every based $G$-connected proper $G$-CW-complex $X$ over $ B$.
\end{theorem}
\begin{proof}
Given a $G$-equivariant parametrized cohomology theory $\calh_G^*$ over $ B$, since the category of proper $G$-CW-complexes has a compact generating set $\mathbb{D}_B^c$, one applies Theorem \ref{abstractbrown} for the parametrized cohomology theory $\calh_G^*$, and we obtain a Brown Representability Theorem for reduced proper $G$-equivariant parametrized cohomology theories.
\end{proof}

We can apply the above theorem for Bredon cohomology associated to a cover. Note that this functor $\IH_G^p(\Phi(-),\mu)$ is an equivariant parametrized cohomology theory over $|\calc|$.  Therefore any operation in cohomology
$$ \mathbb{H}^p_G(\Phi X, \mu )  \to  \mathbb{H}^{p+q}_G(\Phi X,\mu)$$
which is functorial and only depends on the map $\mu:\Phi X \to |\mathcal{C}|$,
must be obtained by a map   
 ${\mathfrak{TR}}_p\to {\mathfrak{TR}}_{p+q}$ of $\OO_G^ P$-spaces over $|{\mathcal{C}} |$.

%
%



\typeout{-------------------------------------- References  ---------------------------------------}
\bibliographystyle{abbrv}
\bibliography{spectral}

\def\cprime{$'$}
\begin{thebibliography}{10}

\bibitem{antonyan}
S.~A. Antonyan and E.~Elfving.
\newblock The equivariant homotopy type of {$G$}-{ANR}'s for proper actions of
  locally compact groups.
\newblock In {\em Algebraic topology---old and new}, volume~85 of {\em Banach
  Center Publ.}, pages 155--178. Polish Acad. Sci. Inst. Math., Warsaw, 2009.

\bibitem{atiyahsegal}
M.~Atiyah and G.~Segal.
\newblock Twisted {$K$}-theory.
\newblock {\em Ukr. Mat. Visn.}, 1(3):287--330, 2004.

\bibitem{Atiyah-Segal2}
M.~Atiyah and G.~Segal.
\newblock Twisted {$K$}-theory and cohomology.
\newblock In {\em Inspired by {S}. {S}. {C}hern}, volume~11 of {\em Nankai
  Tracts Math.}, pages 5--43. World Sci. Publ., Hackensack, NJ, 2006.

\bibitem{AtiyahHirzebruch}
M.~F. Atiyah and F.~Hirzebruch.
\newblock Vector bundles and homogeneous spaces.
\newblock In {\em Proc. {S}ympos. {P}ure {M}ath., {V}ol. {III}}, pages 7--38.
  American Mathematical Society, Providence, R.I., 1961.

\bibitem{AtiyahSinger}
M.~F. Atiyah and I.~M. Singer.
\newblock Index theory for skew-adjoint {F}redholm operators.
\newblock {\em Inst. Hautes \'Etudes Sci. Publ. Math.}, (37):5--26, 1969.

\bibitem{barcenasespinozajoachimuribe}
N.~B{\'a}rcenas, J.~Espinoza, M.~Joachim, and B.~Uribe.
\newblock Universal twist in equivariant {$K$}-theory for proper and discrete
  actions.
\newblock {\em Proc. Lond. Math. Soc. (3)}, 108(5):1313--1350, 2014.

\bibitem{barcenasvelasquez}
N.~Barcenas and M.~Velasquez.
\newblock Twisted equivariant {K}-theory and {K}-homology of $\mathrm{Sl}_{3}(
  \mathbb{Z})$.
\newblock {\em Algebraic and Geometric Topology}, 14(2):823--852, 2014.
\newblock DOI: 10.2140/agt.2014.14.823.

\bibitem{basusen}
S.~Basu and D.~Sen.
\newblock Representing {B}redon cohomology with local coefficients.
\newblock {\em J. Pure Appl. Algebra}, 219(9):3992--4015, 2015.

\bibitem{braun}
V.~Braun.
\newblock Twisted {$K$}-theory of {L}ie groups.
\newblock {\em J. High Energy Phys.}, (3):029, 15 pp. (electronic), 2004.

\bibitem{lueckdavis}
J.~F. Davis and W.~L{\"u}ck.
\newblock Spaces over a category and assembly maps in isomorphism conjectures
  in {$K$}- and {$L$}-theory.
\newblock {\em $K$-Theory}, 15(3):201--252, 1998.

\bibitem{douglas}
C.~L. Douglas.
\newblock On the twisted {$K$}-homology of simple {L}ie groups.
\newblock {\em Topology}, 45(6):955--988, 2006.

\bibitem{Dwyer}
C.~Dwyer.
\newblock Twisted equivariant {$K$}-theory for proper actions of discrete
  groups.
\newblock {\em $K$-Theory}, 38(2):95--111, 2008.

\bibitem{DwyerSpalinski}
W.~G. Dwyer and J.~Spalinski.
\newblock {\em Homotopy theories and model categories}.
\newblock Handbook of algebraic topology, Elsevier, 1995.

\bibitem{elmendorf}
A.~D. Elmendorf.
\newblock Systems of fixed point sets.
\newblock {\em Trans. Amer. Math. Soc.}, 277(1):275--284, 1983.
\newblock MR MR690052 (84f:57029).

\bibitem{freedhopkinsteleman1}
D.~S. Freed, M.~J. Hopkins, and C.~Teleman.
\newblock Loop groups and twisted {$K$}-theory {I}.
\newblock {\em J. Topol.}, 4(4):737--798, 2011.

\bibitem{freedhopkinsteleman3}
D.~S. Freed, M.~J. Hopkins, and C.~Teleman.
\newblock Loop groups and twisted {$K$}-theory {III}.
\newblock {\em Ann. of Math. (2)}, 174(2):947--1007, 2011.

\bibitem{honkasalo}
H.~Honkasalo.
\newblock The equivariant {S}erre spectral sequence as an application of a
  spectral sequence of {S}panier.
\newblock {\em Topology Appl.}, 90(1-3):11--19, 1998.

\bibitem{kneezel}
D.~Kneezel and I.~Kriz.
\newblock Completing {V}erlinde algebras.
\newblock {\em J. Algebra}, 327:126--140, 2011.

\bibitem{lahtinen}
A.~Lahtinen.
\newblock The {A}tiyah-{S}egal completion theorem in twisted {K}-theory.
\newblock {\em Algebraic and Geometric Topology}, 12(4):1925--1940, 2012.
\newblock DOI: 10.2140/agt.2012.12.1925.

\bibitem{degrijselearyvectorbundles}
I.~J. Leary and D.~Degrijse.
\newblock Equivariant vector bundles over classifying spaces for proper
  actions.
\newblock {\em arXiv: 1504.07358 [math.AT]}, 2015.

\bibitem{lueckolivervectorbundles}
W.~L{\"u}ck and B.~Oliver.
\newblock The completion theorem in {K}-theory for proper actions of a discrete
  group.
\newblock {\em Topology}, 40:585--616, 2001.

\bibitem{LueckUribe}
W.~L\"uck and B.~Uribe.
\newblock Equivariant principal bundles and their classifying spaces.
\newblock {\em Algebraic and Geometric Topology}, 14(4):1925--1995, 2014.
\newblock DOI: 10.2140/agt.2014.14.1925.

\bibitem{LupercioUribe}
E.~Lupercio and B.~Uribe.
\newblock Gerbes over orbifolds and twisted {$K$}-theory.
\newblock {\em Comm. Math. Phys.}, 245(3):449--489, 2004.

\bibitem{matumoto}
T.~Matumoto.
\newblock Equivariant {CW} complexes and shape theory.
\newblock {\em Tsukuba J. Math.}, 13(1):157--164, 1989.

\bibitem{maysigurdsson}
J.~P. May and J.~Sigurdsson.
\newblock {\em Parametrized homotopy theory}, volume 132 of {\em Mathematical
  Surveys and Monographs}.
\newblock American Mathematical Society, Providence, RI, 2006.

\bibitem{Segal}
G.~Segal.
\newblock Equivariant {$K$}-theory.
\newblock {\em Inst. Hautes \'Etudes Sci. Publ. Math.}, (34):129--151, 1968.

\bibitem{Witten}
E.~Witten.
\newblock D-branes and {$K$}-theory.
\newblock {\em J. High Energy Phys.}, (12):Paper 19, 41 pp.\ (electronic),
  1998.

\end{thebibliography}
\end{document}